\documentclass[english]{article}
\usepackage[T1]{fontenc}
\usepackage[latin9]{inputenc}
\usepackage[active]{srcltx}
\usepackage{babel}
\usepackage{babelbib}
\usepackage{amsmath}
\usepackage{amsthm}
\usepackage{amsfonts}
\usepackage{amssymb}
\usepackage{sectsty}
\allsectionsfont{\centering \normalfont\scshape}
\usepackage{bbm}
\usepackage{dsfont}
\usepackage{calligra}
\usepackage{calrsfs}
\usepackage{graphicx}
\usepackage[all]{xy}
\usepackage{tikz}
\usetikzlibrary{arrows,shapes}
\usepackage{verbatim}
\usepackage[unicode=true,
 bookmarks=false,
 breaklinks=false,pdfborder={0 0 1},backref=section,colorlinks=false]
 {hyperref}
\usepackage{breakurl}

\theoremstyle{plain}
\newtheorem*{thm*}{\protect\theoremname}
\newtheorem{thm}{\protect\theoremname}
\newtheorem{theorem}{\protect\theoremname}
\newtheorem{lemma}[thm]{\protect\lemmaname}
\newtheorem{corollary}[thm]{\protect\corollaryname}
\newtheorem*{corrthm*}{ZC-Representation Theorem}{\bf}{\it}
\newtheorem{lemma*}{Lemma}{\bf}{\it}
\theoremstyle{definition}
\newtheorem{definition}[thm]{\protect\definitionname}
\newtheorem*{definition*}{Definition}{\bf}

\providecommand{\corollaryname}{Corollary}
\providecommand{\definitionname}{Definition}

\providecommand{\lemmaname}{Lemma}

\providecommand{\theoremname}{Theorem}

\theoremstyle{definition}

\newcommand{\N}{\mathbb{N}}

\newcommand{\Q}{\mathbb{Q}}
\newcommand{\R}{\mathbb{R}}

\usepackage{xspace}
\newcommand{\ignore}[1]{}
\newcommand{\dapm}{DAPM\xspace}
\newcommand{\dapms}{DAPMs\xspace}
\newcommand{\gapm}{GAPM\xspace}
\newcommand{\gapms}{GAPMs\xspace}

\newcommand{\dampcat}{$R$\mbox{-}\textbf{Persmod}}
\newcommand{\rtmodcat}{$R[t]$\mbox{-}\textbf{Gr-Mod}}
\newcommand{\gampcat}{$R$\mbox{-}\textbf{Persmod}^G}
\newcommand{\rgmodcat}{$R[$G$]$\mbox{-}\textbf{Gr-Mod}}

\newcommand{\rtmodname}{\mathbf{M}}
\newcommand{\rgmodname}{\mathbf{M}}

\newcommand{\dapmname}{\mathcal{M}}
\newcommand{\otherdapmname}{\mathcal{N}}
\newcommand{\gapmnameNk}{\mathcal{M}_{\N^k}}
\newcommand{\gapmname}{\mathcal{M}_G}
\newcommand{\othergapmname}{\mathcal{N}_G}
\newcommand{\fpmapping}{\mu}

\begin{document}

\title{The Representation Theorem of Persistence Revisited and Generalized}

\author{Ren\'e Corbet\thanks{Graz University of Technology, corbet@tugraz.at }~~~~~~
Michael Kerber\thanks{Graz University of Technology, kerber@tugraz.at }}
\maketitle

\begin{abstract}
The Representation Theorem by Zomorodian and Carlsson has been the
starting point of the study of persistent homology under the lens
of representation theory. In this work, we give a more accurate
statement of the original theorem and provide a complete and self-contained
proof. Furthermore, we generalize the statement from the case of linear
sequences of $R$-modules to $R$-modules indexed over more ge\-neral
monoids. This generalization subsumes the Representation Theorem of
multidimensional persistence as a special case. 
\end{abstract}

\section{Introduction}

\label{sec:introduction}
\emph{Persistent homology}, introduced by Edelsbrunner et al.~\cite{elz-topological},
is a multi-scale extension of classical homology theory. The idea
is to track how homological features appear and disappear in a shape
when the scale parameter is increasing. This data can be summarized
by a \emph{barcode} where each bar corresponds to a homology class
that appears in the process and represents the range of scales where
the class is present. The usefulness of this paradigm in the context
of real-world data sets has led to the term \emph{topological data
analysis}; see the surveys~\cite{carlsson-survey,ghrist-survey,em-persistent,mvj-survey,kerber-survey}
and textbooks~\cite{eh-computational,oudot-book} for various use
cases.

A strong point of persistent homology is that it can be defined and
motivated both in geometric and in algebraic terms. For the latter, the main
object are \emph{persistence modules}. In the simplest case, such
a persistence module consists of a sequence of $R$-modules indexed
over $\mathbb{N}$ and module homomorphisms connecting consecutive
modules, as in the following diagram:

\[
\begin{xy}\xymatrix{M_{0}\ar[r]^{\varphi_{0}} & M_{1}\ar[r]^{\varphi_{1}} & \dots\ar[r]^{\varphi_{i-1}} & M_{i}\ar[r]^{\varphi_{i}} & M_{i+1}\ar[r]^{\varphi_{i+1}} & \dots}
\end{xy}
\]

A persistence module as above is of \emph{finitely generated type}
if each $M_{i}$ is finitely generated and there is an $m\in\mathbb{N}$
such that $\varphi_{i}$ is an isomorphism for all $i\geq m$. Under
this condition, Zomorodian and Carlsson~\cite{ZC05} observed that
a persistence module can be expressed as single module over the polynomial
ring $R[t]$:

\begin{corrthm*}\textit{{[}Theorem 3.1 in \cite{ZC05}{]} Let $R$ be a commutative
ring with unity. The category of persistence modules of finitely generated
type\footnote{In~\cite{ZC05}, the term ``finite type'' is used instead, but
we renamed it here as we will define another finiteness condition
later.} over $R$ is equivalent to the category of finitely generated graded
modules over $R[t]$. }\end{corrthm*}

The importance of this equivalence stems from the case most important
for applications, namely if $R$ is a field. In this case, graded
$R[t]$-modules, and hence also persistence modules of finitely generated
type, permit a decomposition

\[
\left(\bigoplus_{i=1}^{n}\Sigma^{\alpha_{i}}R\left[t\right]\right)\oplus\left(\bigoplus_{j=1}^{m}\Sigma^{\beta_{j}}R\left[t\right]/\left(t^{n_{j}}\right))\right)
\]

where $\Sigma^{\cdot}$ denotes a shift in the grading. The integers
$\alpha_{i},\beta_{j},n_{j}$ give rise to the aforementioned barcode
of the persistence module; see~\cite{ZC05} for details. Subsequent
work studied the property of more general persistence modules, for
instance, for modules indexed over any subset of $\mathbb{R}$ (and
not necessarily of finite type)~\cite{csgo-structure,crawley-boevey-decomposition}
and for the case that the $M_{i}$ and $\varphi_{i}$ are replaced
with any objects and morphisms in a target category~\cite{BDSS13,BS14}.

\medskip{}

Given the importance of the ZC-Representation Theorem, it is remarkable
that a comprehensive proof seems not to be present in the literature. 
In~\cite{ZC05},
the authors assign an $R[t]$-module to a persistence module of finite
type and simply state: 
\begin{quote}
The proof is the Artin-Rees theory in commutative algebra (Eisenbud,
1995). 
\end{quote}
In Zomorodian's textbook~\cite{afra-book}, the same statement is
accompanied with this proof (where $\alpha$ is the assignment mentioned
above): 
\begin{quote}
It is clear that $\alpha$ is functorial. We only need to construct
a functor $\beta$ that carries finitely generated non-negatively
graded $k[t]$-modules [sic] to persistence modules of finite{[}ly generated{]}
type. But this is readily done by sending the graded module $M=\oplus_{i=0}^{\infty}M_{i}$
to the persistence module $\{M_{i},\varphi_{i}\}_{i\in\mathbb{N}}$
where $\varphi_{i}:M^{i}\rightarrow M^{i+1}$ is multiplication by
t. It is clear that $\alpha\beta$ and $\beta\alpha$ are canonically
isomorphic to the corresponding identity functors on both sides. This
proof is the Artin-Rees theory in commutative algebra (Eisenbud, 1995). 
\end{quote}

While that proof strategy works for the most important case of fields, 
it fails for ``sufficiently'' bad
choices of $R$, as the following example shows:

\smallskip{}
Let $R=\mathbb{Z}[x_{1},x_{2},\ldots]$ and consider the graded $R[t]$
module $M:=\oplus_{i\in\mathbb{N}}M_{i}$ with $M_{i}=R/<x_{1},\ldots,x_{i}>$
where multiplication by $t$ corresponds to the map $M_{i}\rightarrow M_{i+1}$
that assigns $p\,\mathrm{mod}\,x_{i}$ to a polynomial $p$. $M$ is generated
by $\left\{ 1\right\} $. However, the persistence module $\beta(M)$
as in Zomorodian's proof is not of finitely generated type, because
no inclusion $M_{i}\rightarrow M_{i+1}$ is an isomorphism.

\smallskip{}

This counterexample raises the question: what are the requirements
on the ring $R$ to make the claimed correspondence valid? In the
light of the cited Artin-Rees theory, it appears natural to require
$R$ to be a Noetherian ring (that is, every ascending chain of ideals
becomes stationary), because the theory is formulated for such rings
only; see \cite{Eis95,GP07}. Indeed, as carefully exposed in the master's
thesis of the first author~\cite{corbet-master}, the above proof
strategy works under the additional assumption of $R$ being Noetherian. 
We sketch the proof in Appendix~\ref{artin-rees-appendix}.

\paragraph{Our contributions.}

As our first result, we prove a generalized version of the ZC-Representation
Theorem. In short, we show that the original statement becomes valid
without additional assumptions on $R$ if ``finitely generated type''
is replaced with ``finitely presented type'' (that is, in particular,
every $M_{i}$ must be finitely presented). 
Furthermore, we remove the requirement of $R$ being commutative
and arrive at the following result.
\begin{thm*}
Let $R$ be a ring with unity. The category of persistence
modules of finitely presented type over $R$ is isomorphic to the
category of finitely presented graded modules over $R[t]$. 
\end{thm*}
The example from above does not violate the statement of this theorem
because the module $M$ is not finitely presented. Also, the statement
implies the ZC-Representation Theorem for commutative Noetherian rings,
because if $R$ is commutative with unity and Noetherian, finitely
generated modules are finitely presented.

Our proof follows the same path as sketched by Zomorodian, using the
functors $\alpha$ and $\beta$ to define a (straight-forward) correspondence
between persistence mo\-dules and graded $R[t]$-modules. The technical
difficulty lies in showing that these functors are well-defined if
restricted to subclasses of finitely presented type. It is worth to
remark that our proof is elementary and self-contained and does not
require Artin-Rees theory at any point. We think that the ZC-Representation
Theorem is of such outstanding importance in the theory of persistent
homology that it deserves a complete proof in the literature.

\smallskip{}

As our second result, we give a Representation Theorem for a more
general class of persistence modules. We work over an arbitrary ring
$R$ with unity and generalize the indexing set of persistence modules
to a monoid $(G,\star)$.\footnote{Recall that a monoid is almost a group,
except that elements might not have inverses.} 
We consider a subclass which we call ``good'' monoids in this work
(see Section~\ref{sec:prelim_monoid} for the definition and a discussion of 
related concepts). Among them is the case $\mathbb{N}^{k}$ 
corresponding to multidimensional persistence modules,
but also other monoids such as $(\Q_{\geq 0},+)$, $(\Q\cap(0,1],\cdot)$ and 
the non-commutative word monoid as illustrated in Figure~\ref{fig:graph_word}. 
It is not difficult to show that such generalized persistence modules can be 
isomorphically described as a single module over the monoid ring $R[G]$.

Our second main result 
is that finitely presented graded modules over $R[G]$ correspond
again to generalized persistence modules with a finiteness condition.
Specifically, finiteness means that there exists a finite set $S$
of indices (i.e., elements in the monoid) such that for each monoid
element $g$ with associated $R$-module $R_{g}$, there exists an
$s\in S$ such that each map $R_{s}\rightarrow R_{\tilde{g}}$ is an 
isomorphism, whenever $\tilde{g}$ lies between $s$ and $g$. 

For $G=\N^k$, we prove that 
this condition is equivalent to the property that all
sequences in our persistence module are of finite type 
(as a persistence module over $\mathbb{N}$), see Figure~\ref{fig:graph_grid}, 
but this equivalence fails for general (good) monoids. 
Particularly, our second main result implies the first one, because
for $G=\N$, the monoid ring $R[\N]$ is precisely the polynomial ring $R[t]$.

\begin{figure}
        \centering
	\tikzset{treenode/.style = {align=center, inner sep=0pt, text centered, font=\sffamily},
			arr/.style = {treenode, circle, white, font=\sffamily\bfseries, draw=black,fill=blue, text width=2.5em}, 
  			arrempty/.style = {treenode, circle, black, font=\sffamily\bfseries, draw=white,fill=white, text width=1.5em,  minimum width=0.5em, minimum height=0.5em, very thick}
	}
	\begin{tikzpicture}[->,>=stealth',level/.style={sibling distance = 5.5cm/#1,
		level distance = 1.2cm}] 

		\node [arr] {$\mathbf{M_e}$}
			child{ node [arr] {$\mathbf{M_a}$} 
          				 child{ node [arr] {$\mathbf{M_{aa}}$} 
          					child{ node [arrempty] {\reflectbox{$\ddots$}$\,\ddots$} 	
					} 
					child{ node [arrempty] {\reflectbox{$\ddots$}$\,\ddots$} }
          					}
   					 child{ node [arr] {$\mathbf{M_{ab}}$}
					child{ node [arrempty] {\reflectbox{$\ddots$}$\,\ddots$} }
					child{ node [arrempty] {\reflectbox{$\ddots$}$\,\ddots$} }
		   	   }                            
		  }
 			 child{ node [arr] {$\mathbf{M_b}$}
			 child{ node [arr] {$\mathbf{M_{ba}}$} 
						child{ node [arrempty] {\reflectbox{$\ddots$}$\,\ddots$} }
						child{ node [arrempty] {\reflectbox{$\ddots$}$\,\ddots$} }
          	  }
          			  child{ node [arr] {$\mathbf{M_{bb}}$}
						child{ node [arrempty] {\reflectbox{$\ddots$}$\,\ddots$} }
						child{ node [arrempty] {\reflectbox{$\ddots$}$\,\ddots$} }
         	   }
	}
;
\end{tikzpicture}
	  \caption{Graphical illustration of a generalized persistence module. The underlying
		monoid is the set of words over $\{a,b\}$. For each monoid element, the persistence
		module contains an $R$-module, and for each arrow, the module contains
		a homomorphism (which is not specified in the figure).}
	\label{fig:graph_word} 
\end{figure}
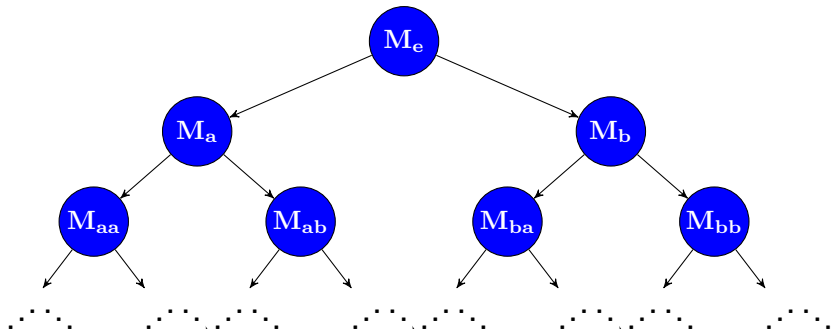

\begin{figure}
        \centering
\begin{tikzpicture}
	\tikzstyle{grid lines}=[lightgray,line width=0]
	\draw[style=grid lines, step=1cm] (0,0) grid (11.5,5.5);
	\foreach \r in {0,1,..., 11}
		\draw (\r,0) node[inner sep=1pt,below=2pt,rectangle,fill=white] {\fontsize{4}{7}$\r$};
	\foreach \r in {0,1, 2,..., 5}
		\draw (0,\r) node[inner sep=1pt,left=2pt,rectangle,fill=white] {\fontsize{4}{7}$\r$};
	\tikzstyle{every node}=[circle, draw, fill=blue, inner sep=0pt, minimum width=4pt]
	\draw[gray] (0,0)  -- (0,5.5) ;
	\draw[gray]  (0,0)  -- (11.5,0) ;
	\draw[shorten >=2pt,->,>=stealth'] (0,0) node {}  -- (2,3) node {};
	\draw[shorten >=2pt,->,>=stealth'] (2,3) node {}  -- (5,4) node {};
	\draw[shorten >=2pt,->,>=stealth'] (5,4) node {}  -- (7,4) node {};
	\draw[shorten >=2pt,->,>=stealth'] (7,4) node {}  -- (7,5) node {};
	\draw (7,5) node {}  -- (8,5.5) ;
	\draw[shorten >=2pt,->,>=stealth'] (0,0) node {}  -- (2,1) node {};
	\draw[shorten >=2pt,->,>=stealth'] (2,1) node {}  -- (3,1) node {};
	\draw[shorten >=2pt,->,>=stealth'] (3,1) node {}  -- (6,2) node {};
	\draw[shorten >=2pt,->,>=stealth'] (6,2) node {}  -- (10,3) node {};
	\draw (10,3) node {}  -- (11.5,4.5) ;

\end{tikzpicture}
     
	  \caption{Graphical illustration of two sequences in a generalized persistence module over the monoid $\mathbb{N}^2$. The corresponding $R$-modules and homomorphisms are not specified in the figure.}
	\label{fig:graph_grid} 
\end{figure}
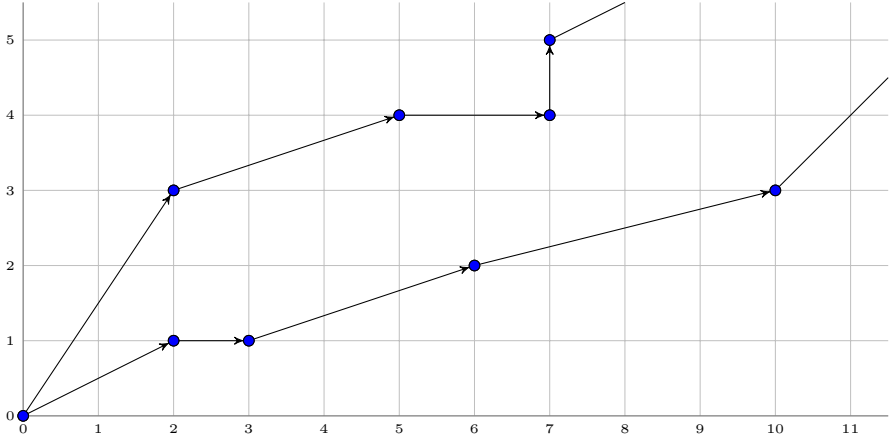

\paragraph{Outline.}

Although our first main result is a special case of the second one,
we decided to give a complete treatment of the classical case of linear
sequences first. For that, we introduce the necessary basic concepts
in Section~\ref{sec:basic_notions} and prove the Representation
Theorem in Section~\ref{sec:representation_over_N}. The additional
concepts required for the monoidal case are introduced in Section~\ref{sec:prelim_monoid}.
The Generalized Representation Theorem is proved in Section~\ref{sec:representation_over_G}.
We conclude in Section~\ref{sec:conclusion}.

\section{Basic notions }

\label{sec:basic_notions}

\paragraph{Category theory.}

We review some basic concepts of category theory as needed in this
exposition. See~\cite{AHS09,ML71} for comprehensive introductions.
A \emph{category} ${\bf C}$ is a collection of \emph{objects} and
\emph{morphisms}, which have to satisfy associativity and identity
axioms: for all morphisms $\alpha\in\text{hom}\left(W,X\right)$,
$\beta\in\text{hom}\left(X,Y\right)$, $\gamma\in\text{hom}\left(Y,Z\right)$
there are compositions $\beta\circ\alpha\in\text{hom}\left(W,Y\right)$
and $\gamma\circ\beta\in\text{hom}\left(X,Z\right)$  which ensure
$\left(\gamma\circ\beta\right)\circ\alpha=\gamma\circ\left(\beta\circ\alpha\right)$.
If it is not clear which category the morphisms belong to, one
writes $\text{hom}_{{\bf C}}$ to denote the category ${\bf C}$.
Furthermore, for all objects $X$ in a category ${\bf C}$ there are
identity morphisms $1_{X}\in\text{hom}\left(X,X\right)$ such that
$1_{X}\circ\alpha=\alpha$ and $\beta\circ1_{X}=\beta$ for all $\alpha\in\text{hom}\left(W,X\right)$,
$\beta\in\text{hom}\left(X,Y\right)$. One example of a category is
the collection of all sets as objects and maps between sets as morphisms.
Others are topological spaces and continuous maps, groups and group
homomorphisms, and vector spaces over a fixed field and linear maps.

A \emph{functor} $F:{\bf C}\rightarrow{\bf D}$ carries the information
from one category to another. It says that $F\left(\beta\circ\alpha\right)=F\left(\beta\right)\circ F\left(\alpha\right)$
and $F\left(1_{X}\right)=1_{F\left(X\right)}$ for all $\alpha\in\text{hom}_{{\bf C}}\left(W,X\right)$,
$\beta\in\text{hom}_{{\bf C}}\left(X,Y\right)$. It suffices to
define functors on the morphisms only, but usually one also specifies
$F\left(X\right)$ for clarity. A simple example of a functor
is the \emph{identity functor} of a category, which simply maps each
object and each morphism on itself. As another example, homology is
a functor from the category of topological spaces to the category
of abelian groups.

Functors describe similarities between
two categories. Two categories ${\bf C},{\bf D}$ are called \emph{isomorphic}
if there are functors $F:{\bf C}\rightarrow{\bf D}$, $G:{\bf D}\rightarrow{\bf C}$
such that $F\circ G$ is the identity functor on ${\bf D}$ and $G\circ F$
is the identity functor on ${\bf C}$. In this case, we also call
the pair $(F,G)$ of functors an \emph{isomorphic pair}.

\paragraph{Graded rings and modules.}

We only consider rings with unity and usually denote them by $R$. A \emph{left-ideal} $I$ of $R$ is an additive subgroup of $R$ such
that $r\in R$ and $x\in I$ implies that $rx\in I$. Replacing
$rx$ with $xr$ defines a \emph{right-ideal}. A subgroup $I$ is
called \emph{ideal }if it is a left-ideal and a right-ideal.

An $R$-(left-)module $\left(M,+,\cdot\right)$ is an abelian group with
a scalar multiplication (from the left), which is a bi-additive group action of $R$
on it. We usually denote modules by $M$.
For example, every left-ideal of $R$ is also an $R$-module. An \emph{$R$-module
morphism} be\-tween $R$-modules $M$ and $N$ is a group homomorphism
$f:M\rightarrow N$ that also satisfies $f(rx)=rf(x)$ for all $r\in R$
and $x\in M$.

A ring $S$ is \emph{$\N$-graded}, or just \emph{graded} if $S$
can be written as $S=\oplus_{i\in\mathbb{N}}S_{i}$ where each $S_{i}$
is an abelian group and $s_{i}\cdot s_{j}\in S_{i+j}$ whenever $s_{i}\in S_{i}$
and $s_{j}\in S_{j}$. If $R$ is a ring,
the polynomial ring $R[t]$ is naturally graded with $R[t]_{i}$ being
the $R$-module generated by $t^i$. While $R[t]$ might permit
different gradings (e.g. if $R$ itself is a polynomial ring), we
will always assume that $R[t]$ is graded in the above way.

If $S=\oplus S_{i}$ is a graded ring, an $S$-module $M$ is \emph{graded}
if there is a decomposition $M=\oplus_{i\in\mathbb{N}}M_{i}$ such
that each $M_{i}$ is an abelian group and $s_{i}\cdot m_{j}\in M_{i+j}$
whenever $s_{i}\in S_{i}$ and $m_{j}\in M_{j}$. As a trivial example,
$S$ is a graded $S$-module itself. By definition, every nonzero
element of $M$ can be written as a finite sum 
\[
m=m_{i_{1}}+m_{i_{2}}+\ldots+m_{i_{k}}
\]
with $k\geq1$, $i_{1}<i_{2}<\ldots<i_{k}$, and $m_{i_{j}}\in M_{i_{j}}$.
If $k=1$, $m=m_{i_{1}}$ is called \emph{homogeneous of degree $i_{1}$}.

A \emph{graded morphism} $f:M\rightarrow N$ be\-tween graded $S$-modules
is an $S$-module morphism such that $f(M_{i})\subset N_{i}$ for all
$i\in\N$. Fixing a ring $R$, the collection of all graded $R[t]$-modules
together with graded morphisms yields another example of a category,
which we will denote by $\rtmodcat$.

\paragraph{Finiteness conditions for modules.}

An $R$-module $M$ is called \emph{finitely generated} if there exist
finitely many elements $\mathfrak{g}_{1},\ldots,\mathfrak{g}_{n}$ 
in $M$ such that every $x\in M$ can be written as $x=\sum_{i=1}^{n}\lambda_{i}\mathfrak{g}_{i}$
with $\lambda_{i}\in R$. The set $\{\mathfrak{g}_{1},\ldots,\mathfrak{g}_{n}\}$
is called the \emph{generating set} of $M$. Equivalently, $M$ is
finitely generated if and only if there exists a surjective module
morphism 
\[
R^{n}\stackrel{\fpmapping}{\rightarrow}M
\]
where $R^{n}$ is just the free abelian $R$-module with $n$ generators
$\mathfrak{e}_{1},\ldots,\mathfrak{e}_{n}$. If $\fpmapping$ maps
$\mathfrak{e}_{i}$ to $\mathfrak{g}_{i}$, we call $\fpmapping$
\emph{associated} to the generating set $\{\mathfrak{g}_{1},\ldots,\mathfrak{g}_{n}\}$.

In general, $\fpmapping$ is not injective and there are relations
be\-tween the generators (which are also sometimes called \emph{syzygies}).
$M$ is called \emph{finitely presented} if it is finitely generated
and the $R$-module $\ker\,\fpmapping$ is finitely generated as well.
Equivalently, finitely presented means that there exists an exact
sequence 
\[
R^{m}\rightarrow R^{n}\stackrel{\fpmapping}{\rightarrow}M\rightarrow0.
\]
Clearly, finitely presented modules are finitely generated, but the
converse is not true as the example from the introduction shows. Note
that morphisms be\-tween finitely presented modules always imply
morphisms on the corresponding free modules, such that the following
diagram commutes: 
\[
\begin{xy}\xymatrix{R^{m_{1}}\ar[r]^ {}\ar[d]^{\varphi_{R}} & R^{n_{1}}\ar[d]^{\varphi_{G}}\ar[r]^{\mu} & M\ar[d]^{\varphi}\\
R^{m_{2}}\ar[r]^ {} & R^{n_{2}}\ar[r]^{\nu} & N
}
\end{xy}
\]
So $\varphi$ can be represented as a
matrix of two blocks, which describe how ge\-nerators and relations
are changing. For an easy proof, see Lemma 2.1.25 in~\cite{GP07} (which also holds
for non-commuta\-tive rings). 

A ring $R$ is called \emph{(left-/right-)Noetherian} if every (left-/right-)ideal
of $R$ is finitely generated. In particular, every principal ideal domain
is Noetherian. We point out the following important statements on Noetherian rings: 
\begin{lemma}
\label{lem:noethersch} Let $R$ be a Noetherian ring with unity.
Then every finitely generated $R$-module
is finitely presented. If $R$ is also commutative, then $R[t]$ is Noetherian.
\end{lemma}
\begin{proof}
The second part is Hilbert's Basis Theorem (see~\cite{vdw93}, 15.1). For the first part, see~\cite{lam99} Proposition 4.29.  
\end{proof}
We will mostly consider the case of finitely generated/presented graded
modules. So, let $M$ be a finitely generated graded $S$-module (where
$S$ is graded, but not necessarily Noetherian). It is not difficult
to see that $M$ is also generated by a finite set of homogeneous
elements in this case, which we will call a \emph{homogeneous generating
set}. With $\fpmapping:S^{n}\rightarrow M$ associated to the homogeneous
generating set $\{\mathfrak{g}_{1},\ldots,\mathfrak{g}_{n}\}$, we
define a grading on $S^{n}$ by setting $\deg(\mathfrak{e}_{i})$
as the degree of $\mathfrak{g}_{i}$ in $M$ and $\deg(s_{i})=i$
for $s_{i}\in S_{i}$ in the grading $\oplus S_{i}$ of $S$. Then
again, each $x\in M$ decomposes into a finite sum of elements of
pairwise distinct degrees, and we can talk about \emph{homogeneous
elements} of $S^{n}$ accordingly. If $M$ is finitely presented,
the generating set of $\ker\,\fpmapping$ can be chosen with homogeneous
elements as well. 

\section{The ZC-Representation Theorem}

\label{sec:representation_over_N}

\paragraph{Persistence modules and $R[t]$-modules.}

\emph{Persistence modules} are the major object of interest in the
theory of persistent homology. We motivate it with the following typical
example: Given a nested sequence of topological spaces indexed over
the integers 
\[
X_{0}\hookrightarrow X_{1}\hookrightarrow X_{2}\hookrightarrow X_{3}\hookrightarrow X_{4}\cdots ,
\]
then the inclusions maps $X_{i}\rightarrow X_{j}$ induce group homomorphisms
$\varphi_{i,j}:H_{\ast}(X_{i})\rightarrow H_{\ast}(X_{j})$. By functoriality
of homology, $\varphi_{i,i}$ is the identity and $\varphi_{i,j}$
is the composition of $\varphi_{k,k+1}$ for $i\leq k\leq j-1$. The
following definition captures these algebraic properties: 
\begin{definition}
Let $R$ be a ring with unity. A \emph{discrete algebraic
persistence module} ($\dapm$) is a tuple $\dapmname=\left(\left(M_{i}\right)_{i\in\mathbb{N}},\left(\varphi_{i,j}\right)_{i\leq j\in\mathbb{N}}\right)$,
such that $M_{i}$ is an $R$-module, $\varphi_{i,j}:M_{i}\rightarrow M_{j}$
is a module morphism, $\varphi_{i,i}=1_{M_{i}}$ and $\varphi_{i,k}=\varphi_{j,k}\circ\varphi_{i,j}$
for all $i\leq k\leq j$. 
\end{definition}
A $\dapm$ is completely specified by the modules and the morphisms
be\-tween consecutive modules, so it is usually just written as 
\begin{eqnarray}
\dapmname:\begin{xy}\xymatrix{M_{0}\ar[r]^{\varphi_{0}} & M_{1}\ar[r]^{\varphi_{1}} & \dots\ar[r]^{\varphi_{i-1}} & M_{i}\ar[r]^{\varphi_{i}} & M_{i+1}\ar[r]^{\varphi_{i+1}} & \dots}
\end{xy}\label{eqn:dapm}
\end{eqnarray}
where $\varphi_{i}:=\varphi_{i,i+1}$.

$\dapms$ over $R$ are closely related to graded $R[t]$-modules:
indeed, given a $\dapm$ $\dapmname$ as in $(\ref{eqn:dapm})$, we
can associate to it a graded $R[t]$-module by setting 
\begin{eqnarray}
\alpha(\dapmname):=\bigoplus_{i\in\N}M_{i}\label{eqn:alpha_objects}
\end{eqnarray}
where multiplication by $t$ is defined by $t\cdot m_{i}:=\varphi_{i}(m_{i})\in M_{i+1}$
for $m_{i}\in M_{i}$. Vice versa, an $R[t]$-module $\oplus_{i\in\N}M_{i}$
defines a $\dapm$ by 
\begin{eqnarray}
\beta\left(\bigoplus_{i\in\N}M_{i}\right):=M_{0}\stackrel{\varphi_{0}}{\rightarrow}M_{1}\stackrel{\varphi_{1}}{\rightarrow}M_{2}\rightarrow\ldots\label{eqn:beta_objects}
\end{eqnarray}
where the morphisms are just multiplication with $t$,
that is $\varphi_{i}(m_{i}):=t\cdot m_{i}$. 
\begin{definition}
For two $\dapm$s $\left(\left(M_{i}\right)_{i\in\mathbb{N}},\left(\varphi_{i,j}\right)_{i\leq j\in\mathbb{N}}\right)$,
$\left(\left(N_{i}\right)_{i\in\mathbb{N}},\left(\psi_{i,j}\right)_{i\leq j\in\mathbb{N}}\right)$,
a family $\xi_{*}=\left(\xi_{i}:M_{i}\rightarrow N_{i}\right)_{i\in\mathbb{N}}$
of module morphisms is called \emph{discrete algebraic persistence
module morphism} if $\psi_{i,j}\circ\xi_{i}=\xi_{j}\circ\varphi_{i,j}$.
Equivalently, the following diagram commutes:

\[
\begin{xy}\xymatrix{M_{0}\ar[r]^{\varphi_{0}}\ar[d]^{\xi_{0}} & \dots\ar[r]^{\varphi_{i-1}} & M_{i}\ar[r]^{\varphi_{i,}}\ar[d]^{\xi_{i}} & M_{i+1}\ar[d]^{\xi_{i+1}}\ar[r]^{\varphi_{i+1}} & \dots\\
N_{0}\ar[r]^{\psi_{0}} & \dots\ar[r]^{\psi_{i-1}} & N_{i}\ar[r]^{\psi_{i}} & N_{i+1}\ar[r]^{\psi_{i+1}} & \dots
}
\end{xy}
\]
With such morphisms, the class of all $\dapms$ over $R$ becomes a category,
which we call $\dampcat$. 
\end{definition}
\begin{lemma}
\label{lem:cat_equiv_general} The maps $\alpha$ and $\beta$ from
$(\ref{eqn:alpha_objects})$ and $(\ref{eqn:beta_objects})$ extend
to functors be\-tween $\dampcat$ and $\rtmodcat$ which form an
isomorphic pair of functors. In particular, the two categories are
isomorphic. 
\end{lemma}
\begin{proof}
Given a $\dapm$ morphism 
\[
\begin{xy}\xymatrix{M_{0}\ar[r]^{\varphi_{0}}\ar[d]^{\xi_{0}} & \dots\ar[r]^{\varphi_{i-1}} & M_{i}\ar[r]^{\varphi_{i,}}\ar[d]^{\xi_{i}} & M_{i+1}\ar[d]^{\xi_{i+1}}\ar[r]^{\varphi_{i+1}} & \dots\\
N_{0}\ar[r]^{\psi_{0}} & \dots\ar[r]^{\psi_{i-1}} & N_{i}\ar[r]^{\psi_{i}} & N_{i+1}\ar[r]^{\psi_{i+1}} & \dots
}
\end{xy}
\]
be\-tween two $\dapms$ $\dapmname$ and $\otherdapmname$, we define
\[
\alpha(\xi_{\ast}):\bigoplus_{i\in\mathbb{N}}M_{i}\rightarrow\bigoplus_{i\in\mathbb{N}}N_{i},(m_{i})_{i\in\mathbb{N}}\mapsto(\xi_{i}(m_{i}))_{i\in\mathbb{N}}.
\]

It is straight-forward to check that $\alpha(\xi_{\ast})$ is a well-defined
morphism be\-tween $\alpha(\dapmname)$ and $\alpha(\mathcal{N})$ 
and that $\alpha$ has the functorial properties~-- see Appendix~\ref{appendix_a}.

Vice versa, a morphism 
\[
\eta:\bigoplus_{i\in\mathbb{N}}M_{i}\rightarrow\bigoplus_{i\in\mathbb{N}}N_{i}
\]
in $\rtmodcat$ induces a homomorphism $\eta_{i}:M_{i}\rightarrow N_{i}$
for each $i\in\mathbb{N}$, and these induced maps are compatible
with multiplication with $t$. Hence, the diagram 
\[
\begin{xy}\xymatrix{M_{0}\ar[r]^{t}\ar[d]^{\eta_{0}} & \dots\ar[r]^{t} & M_{i}\ar[r]^{t}\ar[d]^{\eta_{i}} & M_{i+1}\ar[d]^{\eta_{i+1}}\ar[r]^{t} & \dots\\
N_{0}\ar[r]^{t} & \dots\ar[r]^{t} & N_{i}\ar[r]^{t} & N_{i+1}\ar[r]^{t} & \dots
}
\end{xy}
\]
commutes and so, setting $\beta(\eta):=(\eta_{0},\eta_{1},\ldots)$
yields a $\dapm$ morphism be\-tween $\beta(\oplus_{i\in\mathbb{N}}M_{i})$
and $\beta(\oplus_{i\in\mathbb{N}}N_{i})$. Again, we defer the 
proof of functoriality of $\beta$ to Appendix~\ref{appendix_a}.

Finally, the construction immediately implies that $\alpha\circ\beta$
equals the identity functor on $\rtmodcat$ and $\beta\circ\alpha$
equals the identity functor on $\dampcat$. 
\end{proof}

\paragraph{Finiteness conditions.}

In the context of computation and classification, it is natural to
impose some finiteness condition on persistence modules, yielding
a full subcategory of $\dampcat$. Restricting the functor $\alpha$
from Lemma~\ref{lem:cat_equiv_general} to this subcategory yields
a corresponding subcategory of $\rtmodcat$. But does the correspondence
established above also carry over the finiteness condition in an appropriate
way? This is the question we study in this subsection. 
\begin{definition}
A $\dapm$ $\dapmname=\left(\left(M_{i}\right)_{i\in\mathbb{N}},\left(\varphi_{i,j}\right)_{i\leq j\in\mathbb{N}}\right)$
is \emph{of finite type} if there is a $D\in\mathbb{N}$ such that
for all $D\leq i\leq j$ the map $\varphi_{i,j}$ is an isomorphism.
$\dapmname$ is called \emph{of finitely presented (generated) type}
if it is of finite type and $M_{i}$ is finitely presented (generated) as an $R$-module 
for all $i\in\mathbb{N}$. 
\end{definition}
We will show next that $\dapms$ of finitely presented type over $R$ are isomorphic
to finitely presented graded $R[t]$-modules using the functors $\alpha$
and $\beta$ above. 
\begin{lemma}
\label{lem:alpha_lemma} If a $\dapm$ $\dapmname=\left(\left(M_{i}\right)_{i\in\mathbb{N}},\left(\varphi_{i,j}\right)_{i\leq j\in\mathbb{N}}\right)$
is of finitely presented type, $\alpha(\dapmname)$ is finitely presented. 
\end{lemma}
\begin{proof}
Let $D\in\N$ be such that for all $D\leq i\leq j$, $\varphi_{i,j}:M_{i}\rightarrow M_{j}$
is an isomorphism. Let $\mathfrak{G}{}_{i}$ be a generating set for
$M_{i}$. We claim that $\bigcup_{i=1}^{D}\mathfrak{G}_{i}$ is a
generating set for $\alpha(\dapmname)$. To see that, it suffices
to show that every homogeneous element in $\alpha(\dapmname)=\oplus M_{i}$
is generated by the union of the $G_{i}$. So, fix $k\in\N$ and $m_{k}$
homogeneous of degree $k$. If $k\leq D$, then $m_{k}$ is generated
by the elements of $\mathfrak{G}_{k}$ by construction. If $k>D$,
we show that $m_{k}$ is generated by $\mathfrak{G}_{D}$. For that,
let $m_{D}:=\varphi_{D,k}^{-1}(m_{k})$ which exists because $\varphi_{D,k}$
is isomorphism. $m_{D}$ is generated by $\mathfrak{G}_{D}$, hence
$m_{k}$ is generated by $\varphi_{D,k}(\mathfrak{G}_{D})$. By construction
of $\alpha$, $\varphi_{D,k}(\mathfrak{G}_{D})=t^{k-D}\mathfrak{G}_{D}$
and since $t^{k-D}$ is a ring element in $R[t]$, $m_{k}$ is generated
by $\mathfrak{G}_{D}$. This shows that $\alpha(\dapmname)$ is finitely
generated.

It remains to show that $\alpha(\dapmname)$ is also finitely presented.
Let $\fpmapping_{i}:R^{n_{i}}\rightarrow M_{i}$
be the generating surjective map that corresponds to $\mathfrak{G}_{i}$.
Writing $n=\sum_{i=1}^{D}n_{i}$, there is a map $\fpmapping:R[t]^{n}\rightarrow\alpha(\dapmname)$
that corresponds to the the generating set $\bigcup_{i=1}^{D}\mathfrak{G}_{i}$.
If $\mathfrak{g}_{i}$ is a generator, we will use the notation $\mathfrak{e}_{i}$
to denote the corresponding generator of $R[t]^{n}$.

We now define a finite set of elements of $\ker\,\fpmapping$. First
of all, let $\mathfrak{Z}_{i}$ be the generating set of $\ker\,\fpmapping_{i}$
for $0\leq i\leq D$. Clearly, all elements of $\mathfrak{Z}_{i}$
are also in $\ker\,\fpmapping$. Moreover, for any $0\leq i<j\leq D$,
and any generator $\mathfrak{g}_{i}$ in $\mathfrak{G}_{i}$ with
$\varphi_{i,j}(\mathfrak{g}_{i})\neq0$, we can write 
\[
\varphi_{i,j}(\mathfrak{g}_{i})=\sum_{v=0}^{n_{j}}\lambda_{v}\mathfrak{g}_{j}^{(v)}
\]
where $\lambda_{v}\in R$ and $\mathfrak{G}_{j}=\{\mathfrak{g}_{j}^{(0)},\ldots,\mathfrak{g}_{j}^{(n_{j})}\}$.
In that case, the corresponding element 
\[
t^{j-i}\mathfrak{e}_{i}-\sum_{\nu=0}^{n_{j}}\lambda_{v}\mathfrak{e}_{j}^{(v)}
\]
is in $\ker\,\fpmapping$. We let $\mathfrak{Z}_{i,j}$ denote the
(finite) set obtained by picking one element as above for each $\mathfrak{g}_{i}$
with $\varphi_{i,j}(\mathfrak{g}_{i})\neq0$. We claim that $\mathfrak{Z}:=\bigcup_{i=0}^{D}\mathfrak{Z}_{i}\cup\bigcup_{0\leq i<j\leq D}\mathfrak{Z}_{i,j}$
generates $\ker\,\fpmapping$:

Fix an element in $x\in\ker\,\fpmapping$, which is of the form 
\[
x=\sum_{\ell}\lambda_{\ell}\mathfrak{e}_{\ell}
\]
with $\lambda_{\ell}\in R[t]$ and $\mathfrak{e}_{\ell}$ a generator
of $R[t]^{n}$. We can assume that $x$ is homogeneous of some degree
$k$. We first consider the case that $k\leq D$ and all $\lambda_{\ell}$
are of degree $0$. Then, all $\mathfrak{e}_{\ell}$ that appear in
$x$ are of the same degree, and hence, their images under $\fpmapping$
are generators of $M_{k}$. It follows that $x$ is generated by the
set $\mathfrak{Z}_{k}$.

Next, we consider the case that $k\leq D$, and some $\lambda_{\ell}$
is of positive degree. Because $x$ is homogeneous, $\lambda_{\ell}$
is then of the form $r_{\ell}t^{d_{\ell}}$ for some $r_{\ell}\in R$
and $d_{\ell}>0$. Since the degree of $\mathfrak{e}_{\ell}$ is $k-d_{\ell}$,
there is an element $\mathfrak{z}_{\ell}$ in $\mathfrak{Z}_{k-d_{\ell},k}$
of the form $\mathfrak{z}_{\ell}=t^{d_{\ell}}\mathfrak{e}_{\ell}-\sum_{v=0}^{n_{\ell}}\tilde{\lambda}_{v}\mathfrak{e}_{k}^{(v)}$
with all $\mathfrak{e}_{k}^{(v)}$ of degree $k$ and each $\tilde{\lambda}_{v}\in R$.
Then in $x-r_{\ell}\mathfrak{z}_{\ell}$ the coefficient of $\mathfrak{e}_{\ell}$
in $x$ is $0$, and we only introduce summands with coefficients
of degree $0$ in $t$.

Iterating this construction for each summand with coefficient of positive
degree, we get an element $x'=x-\sum_{w}r_{w}\mathfrak{z}_{w}$ with
$r_{w}\in R$ and $\mathfrak{z}_{w}$ elements of $\mathfrak{Z}$,
and $x'$ only having coefficient of degree $0$ in $t$. This yields
to $x=x'+\sum_{w}r_{w}\mathfrak{z}_{w}$. Using the first part, it
follows that $x$ is generated by $\mathfrak{Z}$.

Finally, we consider the case that $k>D$. In that case, each $\lambda_{\ell}$
is of degree at least $k-D$, because the maximal degree of $\mathfrak{e}_{\ell}$
is $D$. So, $x=t^{k-D}x'$ with $x'$ homogeneous of degree $D$.
Since $0=\fpmapping(x)=t^{k-D}\fpmapping(x')$, $x'\in\ker\,\fpmapping$
as well. By the second part, $x'$ is hence generated by $\mathfrak{Z}$,
and so is $x$.  
\end{proof}
For the next two lemmas, we fix a finitely presented graded $R[t]$-module
$\rtmodname:=\oplus_{i\in\mathbb{N}}M_{i}$ with a map 
\[
R\left[t\right]^{n}\overset{\fpmapping}{\rightarrow}\rtmodname
\]
such that $\ker\,\fpmapping$ is finitely generated. Moreover, we
let $\mathfrak{G}:=\{\mathfrak{g}^{(1)},\ldots,\mathfrak{g}^{(n)}\}$
denote generators of $\rtmodname$ and $\mathfrak{Z}:=\{\mathfrak{z}^{(1)},\ldots,\mathfrak{z}^{(m)}\}$
denote a generating set of $\ker\,\fpmapping$. We assume that each
$\mathfrak{g}^{(i)}$ and each $\mathfrak{z}^{(j)}$ is homogeneous
(with respect to the grading of the corresponding module), and we
let $\deg(\mathfrak{g}^{(i)})$, $\deg(\mathfrak{z}^{(j)})$ denote
the degrees. We further assume that $\mathfrak{G}$ and $\mathfrak{Z}$
are sorted by degrees in non-decreasing order. 
\begin{lemma}
\label{lem:comp_are_finitely_presented} Each $M_{i}$ is finitely
presented as an $R$-module. 
\end{lemma}
\begin{proof}
We argue first that $M_{i}$ is finitely generated. Set $d_{j}:=\deg(\mathfrak{g}^{(j)})$
for $1\leq j\leq n$. Let $n_{i}$ denote the number of elements in
$\mathfrak{G}$ with degree at most $i$. Define the map $\fpmapping_{i}:R^{n_{i}}\rightarrow M_{i}$,
by mapping the $j$th generator $\mathfrak{e}_{i}^{(j)}$ of $R^{n_{i}}$
to the element $t^{i-d_{j}}\mathfrak{g}^{(j)}$. It is then straight-forward
to see that the map $\fpmapping_{i}$ is surjective, proving that
$M_{i}$ is finitely generated.

We show that $\ker\,\fpmapping_{i}$ is finitely generated as well.
Let $\mathfrak{e}^{(1)},\ldots,\mathfrak{e}^{(n)}$ be the generators
of $R[t]^{n}$ mapping to $\mathfrak{g}^{(1)},\ldots,\mathfrak{g}^{(n)}$
under $\fpmapping$. Let $m_{i}$ denote the number of elements in
$\mathfrak{Z}$ such that $d_{j}':=\deg(\mathfrak{z}^{(j)})\leq i$.
For every $\mathfrak{z}^{(j)}$ with $1\leq j\leq m_{i}$, consider
$t^{i-d_{j}'}\mathfrak{z}^{(j)}$, which can be written as 
\[
t^{i-d_{j}'}\mathfrak{z}^{(j)}=\sum_{k=1}^{n_{i}}r_{k}t^{i-d_{k}}\mathfrak{e}^{(k)}
\]
with $r_{k}\in R$. Now, define 
\[
\mathfrak{z}_{i}^{(j)}:=\sum_{k=1}^{n_{i}}r_{k}\mathfrak{e}_{i}^{(k)}
\]
and define $\mathfrak{Z}_{i}:=\{\mathfrak{z}_{i}^{(j)}\mid1\leq i\leq m_{i}\}$.
We claim that $\mathfrak{Z}_{i}$ generates $\ker\,\fpmapping{}_{i}$.
First of all, it is clear that $\fpmapping_{i}(\mathfrak{z}_{i}^{(j)})=\fpmapping(\mathfrak{z}^{(j)})=0$.
Now fix $x\in\ker\,\fpmapping_{i}$ arbitrarily. Then, $x$ is a linear
combination of elements in $\{\mathfrak{e}_{i}^{(1)},\ldots,\mathfrak{e}_{i}^{(n_{i})}\}$
with coefficients in $R$. Replacing $\mathfrak{e}_{i}^{(j)}$ with
$t^{i-d_{j}}\mathfrak{e}^{(j)}$, we obtain $x'\in R[t]^{n}$ homogeneous
of degree $i$. By assumption, we can write $x'$ as linear combination
of elements in $\mathfrak{Z}$, that~is, 
\[
x'=\sum_{k=1}^{m_{i}}r'_{k}t^{i-d_{k}'}\mathfrak{z}^{(k)}
\]
with $r'_{k}\in R$. Then, it holds that 
\[
x=\sum_{k=1}^{m_{i}}r'_{k}\mathfrak{z}_{i}^{(k)},
\]
which follows simply by comparing coefficients: let $j\in\{1,\ldots,n_{i}\}$
and let $c_{j}\in R$ be the coefficient of $\mathfrak{e}_{i}^{(j)}$
in $x$. Let $c'_{j}$ be the coefficient of $\mathfrak{e}_{i}^{(j)}$
in the sum $\sum_{k=1}^{m_{i}}r'_{k}\mathfrak{z}_{i}^{(k)}$, expanding
each $\mathfrak{z}_{i}^{(k)}$ by its linear combination as above.
Then by construction, $c_{j}$ is the coefficient of $t^{i-d_{j}}\mathfrak{e}^{(j)}$
in $x'$, and $c'_{j}$ is the coefficient of $t^{i-d_{j}}\mathfrak{e}^{(j)}$
in the sum $\sum_{k=1}^{m_{i}}r'_{k}t^{i-d_{k}'}\mathfrak{z}^{(k)}$.
Since this sum equals $x'$, it follows that $c_{j}=c'_{j}$. Since
$x$ was chosen arbitrary from $\ker\,\fpmapping_{i}$, it follows
$\mathfrak{Z}_{i}$ generates the kernel. 
\end{proof}
\begin{lemma}
\label{lem:beta_lemma} $\beta(\rtmodname)$ is of finite type. In
particular, it is of finitely presented type with Lemma~\ref{lem:comp_are_finitely_presented}. 
\end{lemma}
\begin{proof}
Fix 
\[
D:=max\{\deg(\mathfrak{g}^{(j)}),\deg(\mathfrak{z}^{(k)})\mid1\leq j\leq n,1\leq k\leq m\}.
\]
It suffices to show that multiplication by $t$ induces an isomorphism
$M_{i}\rightarrow M_{i+1}$ for every $i\geq D$. Let $y\in M_{i+1}$.
Then, $y=\sum_{j=1}^{n}\lambda_{j}\mathfrak{g}^{(j)}$ with $\lambda_{j}\in R[t]$
of degree at least $1$. Hence, $y=ty'$ with $y'\in M_{i}$, showing
that multiplication with $t$ gives a surjective map.

For injectivity, let $y\in M_{i}$ such that $ty=0$. Let $x\in R[t]^{n}$
be such that $\fpmapping(x)=y$. Then $\fpmapping(tx)=ty=0$. Hence,
$tx$ can be written as 
\[
tx=\sum_{j=0}^{m}\tilde{\lambda}_{j}\mathfrak{z}^{(j)}
\]
where each non-trivial $\lambda_{j}$ is a polynomial of degree at
least one, because each $\mathfrak{z}^{(j)}$ is of degree at most
$D$ and $tx$ is of degree at least $D+1$. Therefore, there is also
a decomposition 
\[
tx=\sum_{j=0}^{m}t\lambda_{j}\mathfrak{z}^{(j)}=t\sum_{j=0}^{m}\lambda_{j}\mathfrak{z}^{(j)}.
\]
Since $R[t]^{n}$ is free, this implies that $x$ equals the sum on
the right hand side, implying that $x\in\ker\,\fpmapping$, so $y=0$. 
\end{proof}

\paragraph{The Representation Theorem.}

The preceding lemmas of this section immediately reply the following
version of the Representation Theorem. 
\begin{theorem}
\label{thm:representation_theorem_naturals}
Let $R$ be a ring with unity. The category of finitely presented graded $R[t]$-modules is isomorphic
to the category of discrete algebraic persistence modules of finitely
presented type. 
\end{theorem}
\begin{proof}
The two categories are subcategories of $\rtmodcat$ and $\dampcat$,
respectively. Since $\alpha$ and $\beta$, restricted to these subcategories,
map a $\dapm$ of finitely presented type to a finitely presented graded 
$R[t]$-module (Lemma~\ref{lem:alpha_lemma}) and vice versa
(Lemma~\ref{lem:comp_are_finitely_presented} and Lemma~\ref{lem:beta_lemma}),
these categories are isomorphic. 
\end{proof}
What happens if we replace ``finitely presented'' with the weaker
condition ``finitely generated'' throughout? The proof of Lemma~\ref{lem:alpha_lemma}
shows that if $\dapmname$ is of finitely generated type, $\alpha(\dapmname)$
is finitely generated. Vice versa, if a graded $R[t]$-module $\rtmodname=\oplus_{i\in\mathbb{N}}M_{i}$
is finitely generated, each $M_{i}$ is finitely generated, too. However,
it does not follow in general that $\beta(\rtmodname)$ is of finite
type, as the example from the introduction shows. This problem disappears
with additional requirements on the ring: 
\begin{corollary}
If $R$ is commutative and Noetherian, the category of finitely generated graded $R[t]$-modules
is isomorphic to the category of discrete algebraic persistence modules
of finitely generated type. 
\end{corollary}
\begin{proof}
By Lemma \ref{lem:noethersch}, the Corollary is just Theorem \ref{thm:representation_theorem_naturals} restated. 
\end{proof}

In Appendix~\ref{artin-rees-appendix}, we sketch an alternative proof of
this statement using Artin-Rees theory.

\section{Preliminaries on monoid structures}

\label{sec:prelim_monoid}

A \emph{monoid} is a set $G$ with a binary associative operation
$\star$ and a neutral element $e$, i.e., for all $g_{1},g_{2},g_{3}\in G$
we have $g_{1}\star g_{2}\in G$, $\left(g_{1}\star g_{2}\right)\star g_{3}=g_{1}\star\left(g_{2}\star g_{3}\right)$
and $e\star g_{1}=g_{1}\star e=g_{1}$.  A monoid is called \emph{commutative}
if $g_{1}\star g_{2}=g_{2}\star g_{1}$ for all $g_{1},g_{2}\in G$. We usually denote a monoid
by $G$ or $\left(G,\star\right)$. We sometimes omit denoting $\star$ if it clarifies the expression and no confusion is possible.

\paragraph{Monoid rings and gradings. }

Let $\left(G,\star\right)$ be a monoid and $R$ be a ring with unity. The 
\emph{monoid ring $R\left[G\right]$ }is defined as free $R$-module with 
basis $\{X^h\}_{h \in H}$ where ring multiplication of $R\left[G\right]$ is induced by 
its scalar multiplication as $R$-module and $X^{h_1} X^{h_2} := X^{h_1\star h_2}$. 
It can be easily verified that with these operations $R[G]$ indeed
becomes a ring with unity $X^e$.  
Note that both associativity of $\star$  and the existence of a neutral
element in $G$ are required to guarantee a ring structure with unity
on $R[G]$. Moreover, $R[G]$ is commutative if and only if both $R$
and $G$ are commutative. Furthermore, one can easily verify that $aX^g=X^ga$
in $R\left[G\right]$ whenever $a\in R$ and $g\in G$. 

Monoid rings are a generalization of polynomial rings. For instance,
with $R$ a ring, every $a\in R[t]$ has the form $\sum_{n\in\N}a_{n}t^n$
where almost all $a_{n}=0$. By canonical identification of $t^{n}$
with $X^{n}\in R\left[\mathbb{N}\right]$,
the two notions are isomorphic. Completely analogously,
one can obtain $R\left[\mathbb{N}^{n}\right]\cong R\left[t_{1},...,t_{n}\right]$ for all $0\neq n\in\mathbb{N}$. 
In general, we will often write elements of $R[G]$ in the form $\sum_{g\in G}a_{g}g$
with $a_{g}\in R$ and almost all $a_{g}=0$.

The concept of gradings extends from natural numbers to arbitrary
monoids without problems: a ring $S$ is called \emph{$G$-graded
ring} if $S=\oplus_{g\in G}S_{g}$ as abelian groups and $S_{g_{1}}S_{g_{2}}\subseteq S_{g_{1}\star g_{2}}$
for all $g_{1},g_{2}\in G$. Monoid rings $R[G]$ are typical examples for $G$-graded rings. For a $G$-graded ring $S$, an $S$-module
$M$ is called \emph{$G$-graded module} if $M=\oplus_{g\in G}M_{g}$
as abelian groups and $S_{g_{1}}M_{g_{2}}\subseteq M_{g_{1}\star g_{2}}$
for all $g_{1},g_{2}\in G$. 
A \emph{$G$-graded module morphism} $f:M\rightarrow N$ between $G$-graded $S$-modules is a module morphism  with $f\left(M_{g}\right)\subset N_{g}$
for all  $g\in G$. Equivalently, one can think of $f$
as a family $\left(f_{g}\right)_{g\in G}$ of morphisms on the components
with $f_{g}:M_{g}\rightarrow N_{g}$ for all $g\in G$. A $G$-graded
module morphism is an isomorphism if each $f_{g}$ is an isomorphism.
Given a ring $R$ and a monoid $G$, we denote the category of all $G$-graded modules over $R[G]$
and $G$-graded morphisms between $G$-graded $R[G]$-modules by $\rgmodcat$.

\paragraph{Good monoids.}

The categories $\rgmodcat$ exist without further assumptions on mo\-noids. However, in order to
generalize the Representation Theorem to monoids, we will require a few additional properties on monoids.

A monoid is called \emph{right-cancellative} if $g_{1}\star g_{3}=g_{2}\star g_{3}$
implies $g_{1}=g_{2}$. Similarly, it is called \emph{left-cancellative}
if $g_{1}\star g_{2}=g_{1}\star g_{3}$ implies $g_{2}=g_{3}$. It
is called \emph{cancellative} if it is right-cancellative and left-cancellative.
For commutative monoids these three properties are equivalent and in this case
we simply call such a monoid cancellative. A (left-/right-)cancellative 
monoid $G$ admits (left-/right-)cancellativity on a monoid ring $R[G]$ with respect to multiplication by
a monoid element (from the left/right). 

Examples of non-cancellative monoids are $\left[0,1\right]$ with
multiplication, matrices with multiplication and $\left(\mathbb{Z}_{2},\mathbb{N}\cup\left\{ \infty\right\} \right)$
where $\left(x,n\right)\star\left(y,m\right):=\left(xy,\text{min}\left(n,m\right)\right)$
and $e=\left(0,\infty\right)$. An example for a right-cancellative
monoid that is not left-cancellative is the monoid $\left\{ e,g_{1},g_{2},g_{3}\right\} $
where $e$ is the neutral element and $g_{i}\star g_{j}:=g_{i}$ for
all $i\in\left\{ 1,2,3\right\} $.

We call $g_{2}$ a \emph{(left-)multiple} of $g_{1}$ and write $g_{1}\preceq g_{2}$
if there exists an $h\in G$ such that $h\star g_{1}=g_{2}$. A \emph{proper
multiple} of $g_{1}$ is a multiple $g_{2}$ with $g_{1}\neq g_{2}$,
written as $g_{1}\prec g_{2}$. For a subset $\tilde{G}\subseteq G$
an element $h$ is called \emph{common multiple} of $\tilde{G}$ if $g\preceq h$
for all $g\in\tilde{G}$. We call a common multiple $h$ of $\tilde{G}$  \emph{partially least}, if there is no multiple $h'$ of $\tilde{G}$
such that $h'\prec h$. We write \emph{plcm} for partially least common multiples.
We say that the monoid $G$ is \emph{weak plcm} if for any finite subset
$H\subseteq G$ there are at most finitely many distinct partially least common
multiples of $H$.\footnote{It is called ``weak'' plcm to point out that the plcm does not have
to exist or to be unique}

\tikzstyle{vertex}=[circle, white, font=\sffamily\bfseries,   draw=black,fill=blue, minimum width=3.15em, minimum height=3.15em]
\tikzstyle{empty vertex} = [vertex, black, draw=white, fill=white]
\tikzstyle{edge} = [draw,.,line width=.5pt]
\tikzstyle{pointed edge} = [draw,dotted,.,line width=.5pt]
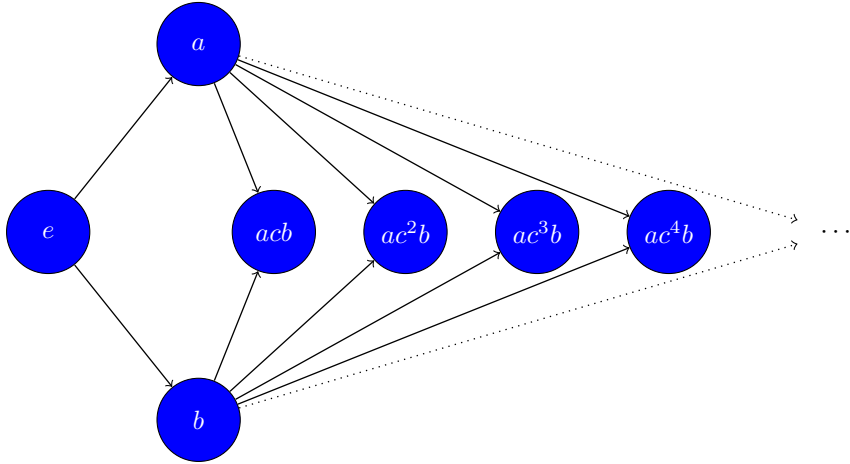
\begin{figure}
\centering
\begin{tikzpicture}[->]
  \foreach \pos/\name in {{(0,0)/e}, {(2,2.5)/a}, {(2,-2.5)/b}, {(3,0)/acb}, {(4.75,0)/ac^2b}, {(6.5,0)/ac^3b}, {(8.25,0)/ac^4b}}                           
        \node[vertex] (\name) at \pos {$\name$};
    \foreach \pos/\name in {{(10.5,0)/dots}}                           
        \node[empty vertex] (\name) at \pos {$\cdots$};
    \foreach \source/ \dest  in {e/a,e/b,a/acb,b/acb,a/ac^2b,a/ac^3b,b/ac^2b,b/ac^3b,b/ac^4b,a/ac^4b}
        \path[edge] (\source) -- (\dest);
        \foreach \source/ \dest  in {a/dots,b/dots}
        \path[pointed edge] (\source) -- (\dest);
\end{tikzpicture}
	  \caption{Graphical illustration of a monoid that is not good. The monoid is the non-commutative monoid generated by three elements $a,b,c$ modulo the congruence generated by $ac^{n}b\approx bc^{n}a$ for all $n\in\N_{>0}$. The plcm of $\{a,b\}$ are the elements of the countable set $\{ac^{n}b\}_{n\in\N_{>0}}$.}
	\label{fig:counterexample} 
\end{figure}

A monoid is \emph{anti-symmetric} if $g_{1}\preceq g_{2}$ and $g_{2}\preceq g_{1}$
imply that $g_{1}=g_{2}$. This is equivalent to the condition that
$\preceq$ turns $G$ into a poset. In an anti-symmetric monoid, no
element (except $e$) can have an inverse element.
\begin{definition}
We call a monoid \emph{good} if it is cancellative, anti-symmetric
and weak plcm. 
\end{definition}
Some easy commutative examples for good monoids with uniquely existing
plcms are $\left(\mathbb{N}^{k},+\right)$, $\left(\mathbb{Q}_{\geq0}^{k},+\right)$,
$\left(\mathbb{R}_{\geq 0}^{k},+\right)$, $\left(\left(0,1\right],\cdot\right)$,
$\left(\mathbb{Q}\cap\left(0,1\right],\cdot\right)$ and $\left(\left[1,\infty\right),\cdot\right)$. 
A fundamental class of good monoids are free monoids, which can be expressed as finite sequences
of elements of a set. In non-commutative free monoids, a subset admits 
common left multiples if and only if it contains a maximal element.
If this exists, it is also the unique plcm. 
Constructing cancellative anti-symmetric monoids that are not weak plcm does not come naturally.
Consider the non-commutative free monoid ge\-ne\-rated by $a,b,c$. Introduce
the congruence generated by $ac^{n}b\approx bc^{n}a$ for all $n\in\mathbb{N}_{>0}$ (see Figure~\ref{fig:counterexample}).
Then there are infinitely many plcm for $\left\{ a,b\right\} $ in
the quotient monoid, since every equivalence class with representative $ac^{n}b$ is partially least for $\left\{ a,b\right\} $.
In particular, good monoids are not closed under homomorphisms.

\tikzstyle{underlying edge} = [draw,line width=.5pt,.,black]
\tikzstyle{a-edge} = [draw,-,ultra thick,line width=4pt,red!90]
\tikzstyle{b-edge} = [draw,line width=4pt,-,yellow!70]
\tikzstyle{c-edge} = [draw,line width=4pt,-,black!15!green]
\tikzstyle{vertex}=[circle, white, font=\sffamily\bfseries,   draw=black,fill=blue, minimum width=1.5em, minimum height=1.5em]
\begin{figure}
\centering
\begin{tikzpicture}[->]
  \foreach \pos/\name in {{(0,0)/e},{(1.89,3.78)/a},{(1.89,0)/b},{(1.89,-2.52)/c}} 
        \node[vertex] (\name) at \pos {$\name$};
   \foreach \pos/\name in {{(11.34,0)/Delta}} 
        \node[vertex] (\name) at \pos {$\Delta$};
    \foreach \pos/\name in {
	{(3.78,3.78)/aa},{(5.67,3.78)/caa},{(7.56,3.78)/acaa},{(9.45,3.78)/bacaa},{(3.78,2.52)/ba},{(5.67,2.52)/cba},{(7.56,2.52)/acba},{(3.78,1.26)/ca},{(5.67,1.26)/aca},{(3.78,0)/cb},{(5.67,0)/acb},{(7.56,0)/aacb},{(9.45,0)/caacb},{(5.67,-1.26)/bac},{(7.56,-1.26)/cbac},{(3.78,-2.52)/ac},{(5.67,-2.52)/aac},{(7.56,-2.52)/caac},{(9.45,-2.52)/acaac}}                           
        \node[vertex] (\name) at \pos {};
    \foreach \source/ \dest  in {e/a,a/aa,caa/acaa,cba/acba,acba/bacaa,ca/aca,cb/acb,acb/aacb,caacb/Delta,cbac/caacb,c/ac,ac/aac,caac/acaac}
        \path[a-edge] (\source) -- (\dest);
    \foreach \source/ \dest  in {e/b,a/ba,acaa/bacaa,aca/aacb,ac/bac,acaac/Delta}
        \path[b-edge] (\source) -- (\dest);
    \foreach \source/ \dest  in {e/c,a/ca,aa/caa,bacaa/Delta,ba/cba,b/cb,aacb/caacb,bac/cbac,aac/caac}
        \path[c-edge] (\source) -- (\dest);
    \foreach \source/ \dest  in {e/a,a/aa,caa/acaa,cba/acba,acba/bacaa,ca/aca,cb/acb,acb/aacb,caacb/Delta,cbac/caacb,c/ac,ac/aac,caac/acaac,e/b,a/ba,acaa/bacaa,aca/aacb,ac/bac,acaac/Delta,e/c,a/ca,aa/caa,bacaa/Delta,ba/cba,b/cb,aacb/caacb,bac/cbac,aac/caac}
        \path[underlying edge] (\source) -- (\dest);
\end{tikzpicture}
	  \caption{Hasse diagram illustrating a finite part of a good monoid. It is the Garside monoid obtained by the free non-commutative monoid  $<a,b,c>$ modulo the congruence generated by $baca\approx a^2cb$, $ca^2cb \approx acbac$ and $acbac\approx cbaca$. The redly, yellowly and greenly highlighted arrows represent left-multiplication with the elements $a$, $b$ and $c$, respectively. A Garside element of a monoid is an element whose left and right divisors coincide, are finite and generate the monoid. ${\Delta}$ is the minimal Garside element. For more details on this example and how such examples can be considered as Garside monoids or divisibility monoids we refer to~\cite{picantin2005}.}
	\label{fig:garsideexample} 
\end{figure}
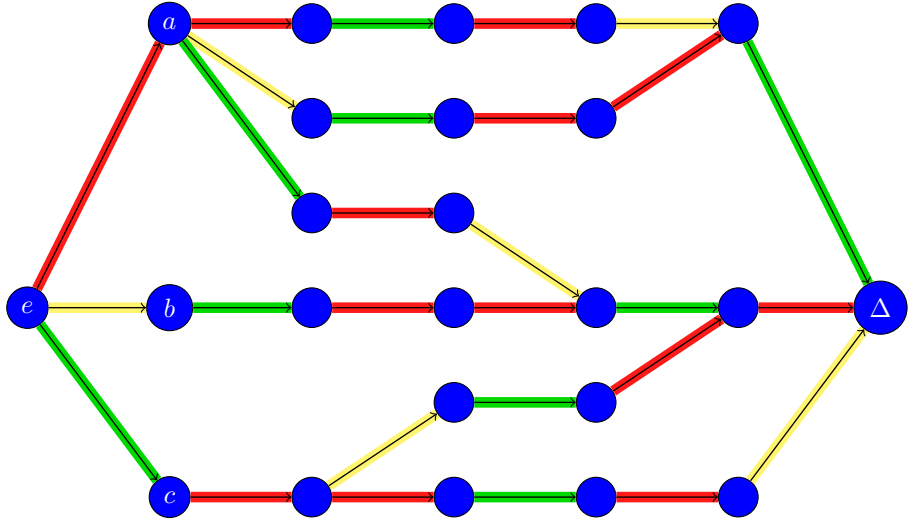

A convenient way to visualize a monoid is a directed multi-graph where each
vertex corresponds to an element of $G$. For a vertex corresponding to
$g\in G$, there is an outgoing edge for each $h\in G$ going to the
vertex $h \star g$. 
We ignore the self loops induced by the neutral element of $G$ on each vertex.
In this interpretation, $g_{1}\preceq g_{2}$ if
and only if there is an edge from the vertex labeled $g_{1}$ to the
vertex labeled $g_{2}$ in the graph. Right-cancellativity means that
there is at most one edge be\-tween any pair of vertices, that is,
the multi-graph is a graph.
Left-cancellativity means that edges obtained by multiplication by
$h$ from pairwise distinct vertices $g_{1},...,g_{n}$ can not lead
to the same vertex $g$. Anti-symmetry simply implies that the graph
is acyclic (modulo self-loops).
Weak plcm means that each finite subset of vertices has
a finite set of minimal common successors. As a consequence, a good
monoid is either trivial or infinite, because if $G$ has at least
two elements, there are at least two outgoing edges per vertex, and
one of them cannot be a self-loop because of right-cancellativity.
Because of acyclicity, we can thus form an infinite sequence of elements. 
For concrete pictures, it is usually convenient to draw only a subset of vertices and edges. 

\paragraph{Related concepts.}

The cancellation property is a classical assumption~\cite{ClPr61} and is sometimes even
part of the definition of a monoid~\cite{GerHK06}. The property of a monoid being 
weak plcm is less standard. It gives rise to a connection to the vivid branch of factorization 
theory~\cite{BaeSme15,GerHK06}. Least common multiples are defined in a related way, 
namely that a (left-~)lcm of a set of elements of a monoid is a right-divisor of all other common multiples. 
If a monoid is cancellative and anti-symmetric, then any set of lcms consists of at most one element. 
Therefore, monoids with these assumptions or with uniquely existing lcms are sometimes called 
lcm monoids~\cite{dehornoy2000braids}. Note that the existence of a plcm does not imply the existence of an lcm.
Conversely, if an lcm exists and is unique, it is also the unique plcm.
Therefore, divisibility monoids~\cite{Ku01} Garside monoids~\cite{dehornoy2015foundations} and Gaussian
monoids~\cite{DehPar99} are examples for subclasses of good monoids. These subclasses are involved in trace 
theory~\cite{DK99} and braid theory \cite{dehornoy2000braids,DehPar99,dehornoy2015foundations}, respectively. 
See Figure~\ref {fig:garsideexample} for a concrete example of a Garside monoid. 
Anti-symmetric monoids are sometimes also called centerless, conical, positive, 
zerosumfree~\cite{wehrung1996}, reduced~\cite{GerHK06} or monoids with trivial unit group~\cite{ClPr61}. 
The preorder induced by left-factorization is related to one of Green's preorders~\cite{grillet2017}. 
Given an anti-symmetric monoid $G$, the aforementioned preorder is a partial order and gives rise to a 
right-partially ordered monoid $(G,\star,\preceq)$, since clearly $g_1\preceq g_2$ implies $h\star g_1\preceq h\star g_2$
for all $h\in G$. If $G$ is commutative, then $(G,\star,\preceq)$ is a partially ordered monoid. A 
formulation of the interleaved equivalence of persistence modules over the totally ordered monoids $\N$ and $\R$  can be found 
in~\cite{VejdJoh12}. Interleaving metrics can even be defined for functors from a fixed preordered set to a fixed arbitrary category~\cite{BDSS13}.

\section{The Representation Theorem over monoids}

\label{sec:representation_over_G} 

\paragraph{Generalized persistence modules and $R\left[G\right]$-modules.}

We will now extend the Representation Theorem from linear sequences
of the form

\[
\begin{xy}\xymatrix{M_{0}\ar[r]^{\varphi_{0}} & M_{1}\ar[r]^{\varphi_{1}} & M_{2}\ar[r]^{\varphi_{2}} & \dots}
\end{xy}
\]

to representations of a monoid $(G,\star,\preceq)$, right-partially ordered by left factorization. For simplicity, we will
assume throughout this section that $(G,\star)$ is a good monoid~\textendash{}
see the conclusion for a discussion of how the conditions of $G$
could be further relaxed. We define persistence modules over $G$: 
\begin{definition}
Let $R$ be a ring with unity. A  \emph{generalized algebraic persistence module} \emph{($\gapm$)} is a tuple $\gapmname=\left(\left(M_{g}\right)_{g\in G},\left(\varphi_{g_{1},g_{2}}\right)_{g_{1}\preceq g_{2}\in G}\right)$ such that $M_{g}$ is an $R$-module, $\varphi_{g_{1,}g_{2}}:M_{g_{1}}\rightarrow M_{g_{2}}$
is a module morphism, $\varphi_{g,g}=1_{M_{g}}$
and $\varphi_{g_{1,}g_{3}}=\varphi_{g_{2,}g_{3}}\circ\varphi_{g_{1,}g_{2}}$
for all $g\in G$, $g_{1}\preceq g_{2}\preceq g_{3}\in G$. 
\end{definition}
Much more succinctly, we could equivalently define a $\gapm$ as a
functor from the poset category $G$ to $R\mbox{-}{\bf Mod}$. It
is clear that a $\gapm$ over the monoid $\mathbb{N}$ is just a $\dapm$.
Of interest is also the case $G=\left(\mathbb{N}^{k},+\right)$, which
has been investigated for example in \cite{ZC09}.

As in the case of $G=\mathbb{N}$, arbitrary $\gapm$ are closely
related to $R\left[G\right]$-modules. Given a $\gapm$ $\gapmname$,
we can assign a graded $R\left[G\right]$-module to it by setting
\begin{eqnarray}
\alpha\left(\gapmname \right):=\bigoplus_{g\in G}M_{g}\label{eqn:alpha_objects-1}
\end{eqnarray}
where multiplication by an element $h\in G$ is defined by $h\cdot m_{g}:=\varphi_{g,h\star g}(m_{g})\in M_{h\star g}$
for all $g\in G$ and all $m_{g}\in M_{g}$. \\ 
Vice versa, an $R[G]$-module $\oplus_{g\in G}M_{g}$ defines a $\gapm$ by 
\begin{eqnarray}
\beta \left( \bigoplus_{g\in G}M_{g} \right)  :=\left(\left(M_{g}\right)_{g\in G},\left(\varphi_{g_{1},g_{2}}\right)_{g_{1}\preceq g_{2}\in G}\right) \label{eqn:beta_objects-1}
\end{eqnarray}
where the morphisms are again defined conversely. More precisely, for all $g_{1}\preceq g_{2}\in G$ and all $m_{g_1}\in M_{g_1}$, we define $\varphi_{g_{1},g_{2}}(m_{g_1}):=h\cdot m_{g_1}$ with
$h\star g_{1}=g_{2}$. Note that $h$ is uniquely defined because
$G$ is assumed to be right-cancellative. 
\begin{definition}
A family of module morphisms $\xi_{G}=\left(\xi_{g}:M_{g}\rightarrow N_{g}\right)_{g\in G}$
be\-tween two $\gapm$ $\left(\left(M_{g}\right)_{g\in G},\left(\varphi_{g_{1},g_{2}}\right)_{g_{1}\preceq g_{2}\in G}\right)$,
$\left(\left(N_{g}\right)_{g\in G},\left(\psi_{g_{1},g_{2}}\right)_{g_{1}\preceq g_{2}\in G}\right)$
over the same ring $R$ with unity is called \emph{generalized algebraic persistence module morphism}
if 
\[
\psi_{g_{1},g_{2}}\circ\xi_{g_{1}}=\xi_{g_{2}}\circ\varphi_{g_{1},g_{2}}
\]
for all $g_{1}\preceq g_{2}\in G$. Equivalently, all diagrams of
the following form commute: 
\[
\begin{xy}\xymatrix{M_{g_{1}}\ar[r]^{\varphi_{g_{1},g_{2}}}\ar[d]^{\xi_{g_{1}}} & M_{g_{2}}\ar[d]^{\xi_{g_{2}}}\\
N_{g_{1}}\ar[r]^{\psi_{g_{1},g_{2}}} & N_{g_{2}}
}
\end{xy}
\]

With such morphisms, the class of all $\gapms$ over $R$ becomes a category,
which we call $\gampcat$. 
\end{definition}
\begin{lemma}
\label{lem:cat_equiv_general-G} The maps $\alpha$ and $\beta$ from
$(\ref{eqn:alpha_objects-1})$ and $(\ref{eqn:beta_objects-1})$ extend
to functors be\-tween $\gampcat$ and $\rgmodcat$ which form an
isomorphic pair of functors. In particular, the two categories are
isomorphic. 
\end{lemma}

The proof is similar to the proof of Lemma \ref{lem:cat_equiv_general}~-- 
see Appendix~\ref{appendix_b} for details.

\paragraph*{Finiteness conditions.}
As in the case of linear sequences, we are interested in the subcategory of $\gampcat$ 
that is isomorphic to the category of finitely presented $R[G]$-modules. In order to describe 
the desired subcategory constructively, we introduce the notions of frames:

\begin{definition}
Let $\gapmname=\left(\left(M_{g}\right)_{g\in G},\left(\varphi_{g_{1},g_{2}}\right)_{g_{1}\preceq g_{2}\in G}\right)$ be a $\gapm$.
For $g\in G$, $h\in G$ is a \emph{frame} of $g$ (wrt.~$\gapmname$)
if $h\preceq g$ and for all $h\preceq \tilde{g} \preceq g$,
the map $\varphi_{h,\tilde{g}}$ is an isomorphism.
We say that $H\subseteq G$ is a \emph{framing set} of $\gapmname$
if every $g\in G$ has a frame in $H$.
$\gapmname$ is of \emph{finite type} if there exists a finite
framing set for $G$. It is called of \emph{finitely presented
(generated) type} if it is of finite type and each $M_{g}$ is finitely
presented (generated) as an $R$-module. 
\end{definition}
From this definition, it follows that $e$ is an element in each framing
set because of anti-symmetry. Moreover, trivially, each $\gapmname$
has a framing set, namely $G$ itself. Another useful property is
that if $g\in G$ and $h$ is a frame of $g$,
then for every $\tilde{g}$ with $h\preceq \tilde{g}\preceq g$,
$\varphi_{\tilde{g},g}$ is an isomorphism as well, just because 
$\varphi_{\tilde{g},g}\circ\varphi_{h,\tilde{g}}=\varphi_{h,g}$.  

We discuss a related concept that is equivalent in some cases. Let
 $\gapmname=\left(\left(M_{g}\right)_{g\in G},\left(\varphi_{g_{1},g_{2}}\right)_{g_{1}\preceq g_{2}\in G}\right)$ 
be a $\gapm$. We define $(g_{i})_{i\in\mathbb{N}}$ to be a sequence in $G$ if $g_i\in G$
and $g_{i}\preceq g_{i+1}$ for all $i\in\N$. A sequence $(g_{i})_{i\in\mathbb{N}}$ in $G$ induces a sequence
$(M_{g_{i}})_{i\in\mathbb{N}}$ in $\gapmname$ with connecting morphisms
$\varphi_{g_{i},g_{i+1}}:M_{g_{i}}\rightarrow M_{g_{i+1}}$. We say that the
latter sequence becomes \emph{stationary} if there exists a $D\in\N$
such that $\varphi_{g_{i},g_{i+1}}$ is an isomorphism for
all $i\geq D$. Note that (stationary) sequences in a \gapm are (finite type) \dapms.
\begin{lemma}
For a $\gapm$ $\gapmname=\left(\left(M_{g}\right)_{g\in G},\left(\varphi_{g_{1},g_{2}}\right)_{g_{1}\preceq g_{2}\in G}\right)$ of finite type, every sequence in $\gapmname$ becomes
statio\-nary. 
\end{lemma}
\begin{proof}
Consider a sequence $(g_{i})_{i\in\mathbb{N}}$ in $G$. Let $H=\{h_{1},\ldots,h_{k}\}$
be a finite set that frames $\gapmname$. Now, for every $g_{i}$,
if $h_{j}\preceq g_{i}$, also $h_{j}\preceq g_{\ell}$ for $i\leq\ell$.
Hence, setting $H_{i}:=\{h\in H\mid h\preceq g_{i}\}$, we get that
$H_{0}\subseteq H_{1}\subseteq\ldots$, and since $H$ is finite,
there is a $D\in\N$ such that $H_{D}=H_{D+1}=\ldots$.

Now, fix some $i\geq D$. Pick a frame $h\in H$ of $g_{i+1}$.
By construction, $h\in H_{D}$ and hence, $h\preceq g_{i}$ as well. 
Hence, $h\preceq g_i\preceq g_{i+1}$, which implies
that $\varphi_{g_{i},g_{i+1}}$ is an isomorphism.
\end{proof}
The converse of the statement is not true in general: Consider the
monoid $\left(\Q_{\geq0},+\right)$. The monoid is good, since the (unique)
plcm of a finite subset is the maximal element in the set. 
We construct a $\gapm$ $\gapmname$
by setting $M_{0}:=R$ and $M_{q}:=0$ for all $q>0$ and let all maps
be the zero maps. It is then obvious that every sequence
becomes stationary. However, no finite subset can frame $\gapmname$
because for every choice of finite subsets $\left\{ 0,q_{1},...,q_{D}\right\} $
with $q_{i}>0$ there are elements $q>0$ such that $q<q_{i}$ for
all $i\in\left\{ 1,...,D\right\} $.

In the previous example, we also observe that the corresponding $R[\Q_{\geq0}]$-module
is finitely generated. In some cases, however, both conditions
are indeed equivalent: 
\begin{lemma}
If $G=\N^{k}$, a $\gapm$ is of finite type if and only if every
sequence becomes stationary. 
\end{lemma}
\begin{proof}
Fix $\gapmnameNk=\left(\left(M_{g}\right)_{g\in \N^{k}},\left(\varphi_{g_{1},g_{2}}\right)_{g_{1}\preceq g_{2}\in \N^{k}}\right)$. Let $H\subset \N^{k}$
be a (not necessarily finite) set framing $\gapmnameNk$.
We call $H$ \emph{reduced} if there is no pair $h_1\neq h_2\in H$,
such that $h_1$ is a frame of $h_2$.
If $H$ is not reduced, and $h_1\preceq h_2$ is a pair as above,
it is not difficult to show that $H\setminus\{h_{2}\}$ is a framing
set for $\gapmnameNk$ as well.
Moreover, for $G=\N^k$, every decreasing sequence of elements
with respect to $\prec$ has a minimal element. 
That implies that there exists a reduced framing set for $\gapmnameNk$.

By the previous lemma, we know that if $\gapmnameNk$ is of finite type,
every sequence becomes stationary. For the converse, if $\gapmnameNk$
is not of finite type, we choose a reduced 
framing set $H\subseteq\N^{k}$, which is necessarily infinite.
We will now construct a non-stationary sequence iteratively,
adding two elements to the sequence in each step.
Set $g_0:=e\in\N^k$ and $H_1:=H\setminus\{g_0\}$;
during the construction, $H_i$ will always be an infinite subset of $H$.
For any such set $H_i$, by Dickson's Lemma (\cite{ros99}, Theorem 5.1), 
there exists a finite
subset $A_{i}\subset H_{i}$ such that for every $h\in H_{i}$, there
exists an $a\in A_{i}$ with $a\preceq h$. 
Define $M_i(a):=\{h\in H_i\mid a\prec h\}$
for $a\in A_{i}$. Since $H_{i}$ is infinite, there exists at least
one $g_{2i}\in A_{i}$ such that $M_i(g_{2i})$ is infinite. 
Since $g_{2i-2}$ is not a frame of $g_{2i}$ 
(because they are both in $H$), there exists
some $g_{2i-1}\in\N^k$ such that
$g_{2i-2}\prec g_{2i-1}\preceq g_{2i}$ and $\varphi_{g_{2i-2},g_{2i-1}}$ is not an
isomorphism. We set $H_{i+1}:=M_i(g_{2i})$ and proceed with the next iteration.
In this way, we obtain a sequence $(g_{i})_{i\in\N}$ 
with $g_{i}\preceq g_{i+1}$ that does not become stationary. 
\end{proof}

We remark that the equivalence is not true in general as the aforementioned
counterexample over $\left(\Q_{\geq0},+\right)$ shows. In this case, a reduced framing set
does not exist. Dually, the monoid $(\N^\N,+)$ of all sequences of natural numbers
yields an example for which the equivalence is not true either, but it is always possible
to obtain reduced framing sets for \gapms indexed over $(\N^\N,+)$.

Let us now see that $\gapms$ of finitely presented type over $R$ are isomorphic
to finitely presented graded $R\left[G\right]$-modules using the
 functors $\alpha$ and $\beta$. The next three lemmas
generalize the corresponding statements in Section~\ref{sec:representation_over_N},
with similar proof ideas. 
\begin{lemma}
\label{lem:alpha_lemma-G}If a $\gapm$ $\gapmname=\left(\left(M_{g}\right)_{g\in G},\left(\varphi_{g_{1},g_{2}}\right)_{g_{1}\preceq g_{2}\in G}\right)$
is of finitely presented type, $\alpha\left(\gapmname\right)$ is
finitely presented. 
\end{lemma}
\begin{proof}
Fix $H=\{h_{1},\ldots,h_{D}\}$ as finite set framing $\gapmname$.
Write $\mathfrak{G}_{i}$ for a finite generating set of $M_{h_{i}}$
and $n_{i}:=|\mathfrak{G}_{i}|$. Let $g\in G$ and $h_{i}$
be a frame of $g$. In particular,
there exists some $f\in G$ such that $f\star h_{i}=g$, meaning that
in $\alpha\left(\gapmname\right)$, the map $f:M_{h_{i}}\rightarrow M_{g}$
is an isomorphism. Hence, the elements of $f\mathfrak{G}_{i}$ generate
$M_{g}$. This shows that $\bigcup_{i=1}^{D}\mathfrak{G}_{i}$
generates $\alpha\left(\gapmname\right)$.

For the second part, let $\fpmapping_{i}:R^{n_{i}}\rightarrow M_{h_{i}}$
be the generating surjective map that corresponds to $\mathfrak{G}_{i}$
where $i\in\left\{ 1,...,D\right\} $. Writing $n=\sum_{i=1}^{D}n_{i}$,
there is a map $\fpmapping:R[G]^{n}\rightarrow\alpha(\dapmname)$
that corresponds to the the generating set $\bigcup_{i=1}^{D}\mathfrak{G}_{i}$.
If $\mathfrak{g}_{i}$ is a generator, we will use the notation $\mathfrak{e}_{i}$
to denote the corresponding generator of $R[G]^{n}$.

We now define a finite set of elements of $\ker\,\fpmapping$. First
of all, let $\mathfrak{Z}_{i}$ be the generating set of $\ker\,\fpmapping_{i}$
for $0\leq i\leq D$. Clearly, all elements of $\mathfrak{Z}_{i}$
are also in $\ker\,\fpmapping$. Moreover, for $h_{i}\preceq h_{j}$
and any generator $\mathfrak{g}_{i}$ in $\mathfrak{G}_{i}$ with
$\varphi_{h_{i},h_{j}}(\mathfrak{g}_{i})\neq0$, we can write 
\[
\varphi_{h_{i},h_{j}}(\mathfrak{g}_{i})=\sum_{v=1}^{n_{j}}\lambda_{v}\mathfrak{g}_{j}^{(v)}
\]
where $\lambda_{v}\in R$ and $\mathfrak{G}_{j}=\{\mathfrak{g}_{j}^{(1)},\ldots,\mathfrak{g}_{j}^{(n_{j})}\}$.
In that case, the corresponding element 
\[
f_{i,j}\mathfrak{e}_{i}-\sum_{\nu=1}^{n_{j}}\lambda_{v}\mathfrak{e}_{j}^{(v)}
\]
is in $\ker\,\fpmapping$ where $f_{i,j}$ is the unique element
such that $f_{i,j}\star h_{i}=h_{j}$.

We let $\mathfrak{Z}_{i,j}$ denote the finite set obtained by picking
one element as above for each $\mathfrak{g}_{i}$ with $\varphi_{h_{i},h_{j}}(\mathfrak{g}_{i})\neq0$.
We claim that $\mathfrak{Z}:=\bigcup_{i=0}^{D}\mathfrak{Z}_{i}\cup\bigcup_{h_{i}\preceq h_{j}}\mathfrak{Z}_{i,j}$
generates $\ker\,\fpmapping$:

Fix an element in $x\in\ker\,\fpmapping$, which is of the form 
\[
x=\sum_{\ell}\lambda_{\ell}\mathfrak{e}_{\ell}
\]
with $\lambda_{\ell}\in R[G]$ and $\mathfrak{e}_{\ell}$ a generator
of $R[G]^{n}$. We can assume that $x$ is homogeneous of some degree
$g$. First we consider the case that $g=h_{j}\in H$ and all $\lambda_{\ell}$
are of degree $e$. Then, all $\mathfrak{e}_{\ell}$ that appear in
$x$ are of the same degree, and hence, their images under $\fpmapping$
are generators of $M_{h_{j}}$. It follows that $x$ is generated
by the set $\mathfrak{Z}_{j}$.

Next, we consider the case that $g=h_{j}\in H$ and some $\lambda_{\ell}$
is of non-trivial degree. Because $x$ is homogeneous, $\lambda_{\ell}$
is then of the form $r_{\ell}f_{\ell,j}$ for some $r_{\ell}\in R$
and non-trivial $f_{\ell,j}\in G$. Since the degree of $\mathfrak{e}_{\ell}$
is a $h_{\ell}$ with $f_{\ell,j}\star h_{\ell}=h_{j}$, there is
an element $\mathfrak{z}_{\ell}$ in $\mathfrak{Z}_{\ell,j}$ of the
form $\mathfrak{z}_{\ell}=f_{\ell,j}\mathfrak{e}_{\ell}-\sum_{v=1}^{n_{\ell}}\tilde{\lambda}_{v}\mathfrak{e}_{j}^{(v)}$
with all $\mathfrak{e}_{j}^{(v)}$ of degree $h_{j}$ and each $\tilde{\lambda}_{v}\in R$.
Then, by turning from $x$ to $x-r_{\ell}\mathfrak{z}_{\ell}$, we
turn the coefficient of $\mathfrak{e}_{\ell}$ into $0$ and we only
introduce summands with coefficients of degree $e$ in any non-trivial
$g\in G$. Iterating this construction for each summand with coefficient
of positive degree, we get an element $x'=x-\sum_{w}r_{w}\mathfrak{z}_{w}$
with $r_{w}\in R$ and $\mathfrak{z}_{w}$ elements of $\mathfrak{Z}$,
and $x'$ only having coefficients of degree $e$ in any non-trivial
$g\in G$. By the first part, $x'$ is generated by $\mathfrak{Z}$,
hence so is $x$.

Finally, let $g\notin H$. Then, there exists some $h_{i}\in H$ and
some $f\in G$ such that $f\star h_{i}=g$ and multiplication with
$f$ yields an isomorphism $M_{h_{i}}\rightarrow M_{g}$. Hence, we
have that $x=fx'$ with $x'$ homogeneous of degree $h_{i}$. Since
$0=\fpmapping(x)=f\fpmapping(x')$, it follows that $x'\in\ker\,\fpmapping$.
By the above cases, $x'$ is generated by $\mathfrak{Z}$, and hence,
so is $x$. 
\end{proof}
For the next two lemmas, we fix a finitely presented graded $R\left[G\right]$-module
$\rgmodname:=\oplus_{g\in G}M_{g}$ with a map 
\[
R\left[G\right]^{n}\overset{\fpmapping}{\rightarrow}\rgmodname
\]
such that $\ker\,\fpmapping$ is finitely generated. Moreover, we
let $\mathfrak{G}:=\{\mathfrak{g}^{(1)},\ldots,\mathfrak{g}^{(n)}\}$
denote generators of $\rtmodname$ and $\mathfrak{Z}:=\{\mathfrak{z}^{(1)},\ldots,\mathfrak{z}^{(m)}\}$
denote a generating set of $\ker\,\fpmapping$. We assume that each
$\mathfrak{g}^{(i)}$ and each $\mathfrak{z}^{(j)}$ is homogeneous
(with respect to the grading of the corresponding module), and we
let $\deg(\mathfrak{g}^{(i)})$, $\deg(\mathfrak{z}^{(j)})$ denote
the degrees. 
\begin{lemma}
\label{lem:comp_are_finitely_presented-G} Each $M_{g}$ is finitely
presented as an $R$-module. 
\end{lemma}
\begin{proof}
We argue first that $M_{g}$ is finitely generated.  Let $n_{g}$ denote the number of elements $\mathfrak{g}^{(j)}$
in $\mathfrak{G}$ such that $\deg(\mathfrak{g}^{(j)})\preceq g$.
Let $h_{j}\in G$ such that $h_{j}\star\deg(\mathfrak{g}^{(j)})=g$.
Define the map $\fpmapping_{g}:R^{n_{g}}\rightarrow M_{g}$, by mapping
the $j$th generator $\mathfrak{e}_{g}^{(j)}$ of $R^{n_{g}}$ to
the element $h_{j}\mathfrak{g}^{(j)}$. Then the map $\fpmapping_{g}$
is surjective, proving that $M_{i}$ is finitely generated.

We show that $\ker\,\fpmapping_{g}$ is finitely generated as well.
Let $\mathfrak{e}^{(1)},\ldots,\mathfrak{e}^{(n)}$ be the gene\-rators
of $R[G]^{n}$ mapping to $\mathfrak{g}^{(1)},\ldots,\mathfrak{g}^{(n)}$
under $\fpmapping$. Let $m_{g}$ denote the number of elements in
$\mathfrak{Z}$ such that $\deg(\mathfrak{z}^{(j)})\preceq g$. Let
$h_{j}'\in G$ such that $h_{j}'\star\deg(\mathfrak{z}^{(j)})=g$.
For every $\mathfrak{z}^{(j)}$ with $1\leq j\leq m_{g}$, consider
$h_{j}'\mathfrak{z}^{(j)}$, which can be written as 
\[
h_{j}'\mathfrak{z}^{(j)}=\sum_{k=1}^{n_{g}}r_{k}h_{k}\mathfrak{e}^{(k)}
\]
with $r^{(k)}\in R$. Now, define 
\[
\mathfrak{z}_{g}^{(j)}:=\sum_{k=1}^{n_{g}}r_{k}\mathfrak{e}_{g}^{(k)}
\]
and $\mathfrak{Z}_{g}:=\{\mathfrak{z}_{g}^{(j)}\mid1\leq i\leq m_{g}\}$.
We claim that $\mathfrak{Z}_{g}$ generates $\ker\,\fpmapping{}_{g}$.
First of all, we get $\fpmapping_{g}(\mathfrak{z}_{g}^{(j)})=\fpmapping(\mathfrak{z}^{(j)})=0$.
Now fix $x\in\ker\,\fpmapping_{g}$ arbitrarily. Then, $x$ is a linear
combination of elements in $\{\mathfrak{e}_{g}^{(1)},\ldots,\mathfrak{e}_{g}^{(n_{g})}\}$
with coefficients in $R$. Replacing $\mathfrak{e}_{g}^{(j)}$ with
$h_{j}\mathfrak{e}^{(j)}$, we obtain $x'\in R[G]^{n}$ homogeneous
of degree $g$. By assumption, we can write $x'$ as linear combination
of elements in $\mathfrak{Z}$, that is, 
\[
x'=\sum_{k=1}^{m_{g}}r'_{k}h_{k}'\mathfrak{z}^{(k)}
\]
with $r'_{k}\in R$. Then, it holds that 
\[
x=\sum_{k=1}^{m_{g}}r'_{k}\mathfrak{z}_{g}^{(k)},
\]
which follows by comparing coefficients: let $j\in\{1,\ldots,n_{g}\}$
and let $c_{j}\in R$ be the coefficient of $\mathfrak{e}_{g}^{(j)}$
in $x$. Let $c'_{j}$ be the coefficient of $\mathfrak{e}_{g}^{(j)}$
in the sum $\sum_{k=1}^{m_{g}}r'_{k}\mathfrak{z}_{g}^{(k)}$, expanding
each $\mathfrak{z}_{g}^{(k)}$ by its linear combination as above.
By construction, $c_{j}$ is the coefficient of $h_{k}\mathfrak{e}^{(j)}$
in $x'$, and $c'_{j}$ is the coefficient of $h_{k}\mathfrak{e}^{(j)}$
in the sum $\sum_{k=1}^{m_{i}}r'_{k}h_{k}'\mathfrak{z}^{(k)}$. Since
this sum equals $x'$, it follows that $c_{j}=c'_{j}$. Since $x$
was chosen arbitrarily from $\ker\,\fpmapping_{g}$, it follows that
$\mathfrak{Z}_{g}$ generates the kernel. 
\end{proof}
\begin{lemma}
\label{lem:beta_lemma-G} $\beta(\rgmodname)$ is of finite type.
In particular, it is of finitely presented type with Lemma~\ref{lem:comp_are_finitely_presented-G}. 
\end{lemma}
\begin{proof}
Define 
\[
D:=\left\{ \deg(\mathfrak{g}^{(j)})  \mid1\leq j\leq n\right\} \cup \left\{\deg(\mathfrak{z}^{(k)})  1\leq k\leq m\right\} .
\]

For every subset $D'$ of $D$, let $plcm(D')$ denote the set of partially least
common multiples of $D'$. Then, set 
\[
H:=\bigcup_{D'\subseteq D}plcm(D').
\]
Note that $e\in H$ (using $D'=\emptyset$) and $D\subset H$ (when
$D'$ ranges over the singleton sets). Also, $H$ is finite because
$D$ is finite and $G$ is weak plcm.

We claim that $H$ frames $\beta(\rgmodname)$. Let $g\in G$ be arbitrary.
We have to find an element $h\in H$ such that $h\preceq g$ and for all $h\preceq \tilde{g}\preceq g$ 
with $\tilde{f}\star h=\tilde{g}$, multiplication with $\tilde{f}$ is an isomorphism $M_{h}\rightarrow M_{\tilde{g}}$.

If $g\in H$, that claim is trivial using $h:=g$ and $\tilde{f}:=e$. So,
let us assume that $g\notin H$. Let $D'(g):=\{\ell\in D\mid\ell\preceq g\}$.
Then, $g$ is a common multiple of $D'(g)$. However, it is not a partially
least common multiple because in that case, it would belong to $H$.
Hence, there exists a plcm $h$ of $D'(g)$ such that $h\preceq g$.
Note also that $D'(g)=D'(h)$ in this case. Let $h\preceq \tilde{g}\preceq g$ and 
$\tilde{f}$ be such that $\tilde{f}\star h=\tilde{g}$.

We first show that multiplication by $\tilde{f}$ is surjective. For that,
fix some $x\in M_{\tilde{g}}$. Then $x$ is generated by a linear combination
of the generators $\mathfrak{g}^{(j)}$, and among those, only those
in $D'(\tilde{g})=D'(g)$ can have non-zero coefficients. Hence, $x$ takes the
form 
\[
x=\sum_{j\in\{1,\ldots,n\},{\deg(\mathfrak{g}^{(j)})\in D'(g)}}f_{j}r_{j}\mathfrak{g}^{(j)}
\]
with $r_{j}\in R$ and $f_{j}\in G$ such that $f_{j}\star\deg(\mathfrak{g}^{(j)})=\tilde{g}$.
Let $f'_{j}$ be such that $f'_{j}\star\deg(\mathfrak{g}^{(j)})=h$.
Then, using $\tilde{f}$ from above, 
\[
\tilde{f}\star f'_{j}\star\deg(\mathfrak{g}^{(j)})=\tilde{g}=f_{j}\star\deg(\mathfrak{g}^{(j)}),
\]
which implies that $\tilde{f}\star f'_{j}=f_{j}$ because $G$ is right-cancellative.
Hence we can write 
\begin{align*}
x & =\sum_{j\in\{1,\ldots,n\},{\deg(\mathfrak{g}^{(j)})\in D'(g)}}(\tilde{f}\star f'_{j})r_{j}\mathfrak{g}^{(j)}\\
 & =\tilde{f}\sum_{j\in\{1,\ldots,n\},{\deg(\mathfrak{g}^{(j)})\in D'(g)}}f'_{j}r_{j}\mathfrak{g}^{(j)},
\end{align*}
which shows that $x$ has a preimage in $M_{h}$.

For the injectivity of multiplication with $\tilde{f}$, let $y\in M_{h}$
such that $\tilde{f}y=0$. Let $x\in R[G]^{n}$ such that $\fpmapping(x)=y$.
Then $\fpmapping(\tilde{f}x)=\tilde{f}y=0$ and hence, $\tilde{f}x$ is generated by the $\mathfrak{z}^{(k)}$.
Similar as in the first part of the proof, only such $\mathfrak{z}^{(k)}$
can appear whose degree lies in $D'(g)$. Hence, writing 
\[
\tilde{f}x=\sum_{k\in\{1,\ldots,m\},{\deg(\mathfrak{z}^{(k)})\in D'(g)}}f_{k}r_{k}\mathfrak{z}^{(k)},
\]
we know that $f_{k}\star\deg(\mathfrak{z}^{(k)})=\tilde{g}$.
The same way as in the first part of the proof, we get that $f_{k}=\tilde{f}\star f'_{k}$
with some $f'_{k}\in G$. It follows that 
\[
\tilde{f}x=\tilde{f}\sum_{k\in\{1,\ldots,m\},{\deg(\mathfrak{z}^{(k)})\in D'(g)}}f'_{k}r_{k}\mathfrak{z}^{(k)}.
\]
Since $G$ is left-cancellative, $R\left[G\right]$ is left-cancellative
with respect to multiplication by any monoid element. This property carries over to the free module $R[G]^n$ and it follows that
\[
x=\sum_{k\in\{1,\ldots,m\},{\deg(\mathfrak{z}^{(k)})\in D'(g)}}f'_{k}r_{k}\mathfrak{z}^{(k)}.
\]
Thus $x$ can be written as linear combination of the $\mathfrak{z}^{(k)}$,
proving that $y=\fpmapping(x)=0$. 
\end{proof}

\paragraph{The Generalized Representation Theorem.}

The preceding lemmas imply the Genera\-lized Representation Theorem: 
\begin{theorem}
\label{thm:representation_theorem_monoids} Let $R$ be a ring with unity. The
category of finitely presented graded $R\left[G\right]$-modules is
isomorphic to the category of generalized algebraic persistence mo\-dules
over $R$ of finitely presented type. 
\end{theorem}
\begin{proof}
The categories are subcategories of $\rgmodcat$ and $\gampcat$,
respectively. Since $\alpha$ and $\beta$, restricted to these subcategories,
map a $\gapm$ of finitely presented type to a finitely presented graded 
$R[G]$-module (Lemma~\ref{lem:alpha_lemma-G}) and vice versa 
(Lemma~\ref{lem:comp_are_finitely_presented-G} and Lemma~\ref{lem:beta_lemma-G}),
these categories are isomorphic. 
\end{proof}
Clearly, this statement contains Theorem~\ref{thm:representation_theorem_naturals}
using $G:=\N$. 

As in Theorem~\ref{thm:representation_theorem_naturals}, replacing
``finitely presented'' with ``finitely generated'' invalidates
the claim. Passing to a Noetherian ring $R$, however, does not immediately
revalidate it, because $R[G]$ is not Noetherian in general. Still,
the following statement follows easily: 
\begin{corollary}
If $R$ and $R\left[G\right]$ are Noetherian rings with unity, the
category of finitely generated graded $R[G]$-modules is isomorphic
to the category of discrete algebraic persistence modules of finitely
generated type. 
\end{corollary}
We remark that Hilbert's Basis Theorem can be applied iteratively (see \cite{vdw93}, 15.1). Using 
$R\left[\mathbb{N}^{k}\right]\cong R\left[t_{1},...,t_{k}\right]\cong R\left[t_{1},...,t_{k-1}\right]\left[t_{k}\right]$, we obtain 
\begin{corollary}
Let $R$ be a commutative Noetherian ring with unity, $k\in\mathbb{N}$.
The category of finitely generated graded $R\left[\mathbb{N}^{k}\right]$-modules
is isomorphic to the category of algebraic persistence modules over
$\mathbb{N}^{k}$ of finitely generated type. 
\end{corollary}

\section{Conclusion}

\label{sec:conclusion} The Representation Theorem for persistence
modules is one of the landmark results in the theory of persistent
homology. In this paper, we formulated a more precise statement of
this classical theorem over arbitrary rings with unity 
where we replaced finite generation with finite
presentation. We provided a proof which only relies
on elementary module theory. Furthermore, we generalized the Representation
Theorem from naturally indexed modules to modules indexed over a more
general class of monoids. The key difficulty was to find the right
finiteness condition for persistence modules in this case to certify
the equivalence with finitely presented modules graded over the monoid. 
Since the underlying ring does not have to be commutative, the Representation 
Theorem should rather be considered as a statement on general 
persistence, including the case of persistent homology.

Alternatively, in categorical language, persistence modules over $\N$
are functors $\mathbb{N}\rightarrow R\mbox{-}{\bf Mod}$
where $\mathbb{N}$ is interpreted as the totally ordered set $\left(\mathbb{\mathbb{N}},\leq\right)$.
For the Representation Theorem we used that $\mathbb{N}$ is a monoid
and endowed the functor $\mathbb{N}\rightarrow R\mbox{-}{\bf Mod}$
with an algebraic structure $R\left[t\right]$. We showed that the
generalization of this situation is replacing $\mathbb{N}$ with a
poset $\left(G,\preceq\right)$ that is induced by left factorization in a good monoid $(G,\star)$.
Moreover, we defined the subcategory of finite-type-functors $G\rightarrow R\mbox{-}{\bf Mod}$
precisely. The concept of frames enables us to locate
the persistent features in the monoid concretely.  

An obvious question is how the finiteness condition changes when the
requirements on the monoid are further relaxed. The anti-symmetry
of the monoid is not a severe assumption because elements $g_{1},g_{2}$
with $g_{1}\preceq g_{2}$ and $g_{2}\preceq g_{1}$ induce isomorphic 
connecting maps in the corresponding $\gapm$ and therefore can be treated 
as one component of the monoid. Also, we can relax the right-cancellative
condition such that between any two elements $g_{1},g_{2}\in G$,
there are at most finitely many $h$ such that $h\star g_{1}=g_{2}$.
We restricted to good monoids for clarity of exposition. If applications occur 
that require filtrations with some complicated ordering structure,
we suggest to consider a good monoid as underlying index set. 

A next step is to find classification results, parametrizations and discrete invariants
for finitely presented monoid-graded modules.  It might be useful to 
consider a subclass of good monoids, e.g., finitely generated good monoids or 
monoids with a manageable factorization theory.
The theory of persistence modules is well-studied in the case of 
$\mathbb{N}^{k}$ and might be transferred to more general monoids, 
as much as the situation of discrete invariants. Since there is no natural notion of 
barcodes in the case of $\mathbb{N}^{k}$, there is also no hope of finding 
barcodes for more general monoids. But it might be possible to generalize a
discrete invariant from $\mathbb{N}^{k}$, for instance the discrete
rank invariant, and generalize algorithms to $G$-indexed persistence. 

\paragraph{Acknowledgments.} The authors are supported by the Austrian Science Fund (FWF) grant number P 29984-N35.

\newpage

\appendix

\section*{Appendix}

\section{Details from Section~\ref{sec:representation_over_N}} \label{appendix_a}

\paragraph{Proof of functoriality of $\alpha$.}

Recall that for $\xi_{\ast}=(\xi_{i})_{i\in\mathbb{N}}$, a $\dapm$
morphism between $\dapmname$ and $\otherdapmname$, we define $\alpha(\xi_{\ast})$
as the map assigning to $(m_{i})_{i\in\mathbb{N}}\in\oplus M_{i}$
the value $(\xi_{i}(m_{i}))_{i\in\mathbb{N}}$. Let us define $f:=\alpha(\xi_{\ast})$
for shorter notation. Indeed, $f$ is a graded module morphism: it
is a group homomorphism, as each $\xi_{i}$ is, it clearly satisfies
$f(M_{i})\subset N_{i}$, and it holds that for $m=(m_{0},m_{1},\ldots)$
we get 
\begin{eqnarray*}
f(tm) & = & f(0,tm_{0},tm_{1},\ldots)=(0,\xi_{1}(tm_{0}),\xi_{2}(tm_{1}),\ldots)\\
 & = & (0,t\xi_{0}(m_{0}),t\xi_{1}(m_{1}),\ldots)=t(\xi_{i}(m_{i}))_{i\in\mathbb{N}}=tf(m),
\end{eqnarray*}
where the third equality comes from the property of $\dapm$ morphisms.

For the functorial properties, it is clear that $\alpha$ maps the
identity $\dapm$ morphism to the identity of the corresponding graded modules. If $\xi_{*}$, $\xi'_{\ast}$ are
$\dapm$ morphisms and $m$ as before, we calculate 
\begin{eqnarray*}
  (\alpha(\xi'_{\ast}\circ\xi_{\ast}))(m) & = & (\xi'_{i}(\xi_{i}(m_{i})))_{i\in\mathbb{N}}\\
 & = & \alpha(\xi'_{\ast})(\xi_{i}(m_{i}))_{i\in\mathbb{N}}\\
 & = & (\alpha(\xi'_{\ast})\circ\alpha(\xi_{\ast}))(m).
\end{eqnarray*}

\paragraph{Proof of functoriality of $\beta$.}

Fix two graded $R[t]$-modules $\oplus M_{i}$ and $\oplus N_{i}$
and let $\dapmname:=\beta(\oplus M_{i})$ and $\otherdapmname:=\beta(\oplus N_{i})$.
A graded morphism $\eta:\oplus M_{i}\rightarrow\oplus N_{i}$ implies
a sequence of maps $(\xi_{i})_{i\in\mathbb{N}}:=\beta(\eta)$ where
$\xi_{i}$ is just the restriction of $\eta$ to $M_{i}$. It follows
that $\beta$ gives group homomorphisms $\xi_{i}:M_{i}\rightarrow N_{i}$
which are also $R$-module morphisms because for $r\in R$, $m_{i}\in M_{i}$
we get $\xi_{i}(rm_{i})=\eta(rm_{i})=r\eta(m_{i})=r\xi_{i}(m_{i})$. Since the
connecting maps in $\dapmname$ and $\otherdapmname$ are induced
by multiplication with $t$, we have that for $m_{i}\in M_{i}$, 
\[
\xi_{i+1}(tm_{i})=\eta(tm_{i})=t\eta(m_{i})=t\xi_{i}(m_{i}),
\]
proving that $\beta(\eta)$ is indeed a $\dapm$ morphism.

For functoriality, it is again clear that $\beta$ maps the identity
to the identity morphism. For two graded $R[t]$-module morphisms
$\eta,\eta'$ and any sequence $(m_{i})_{i\in\mathbb{N}}$ with $m_{i}\in M_{i}$ for all $i\in\N$,
we get 
\begin{eqnarray*}
  (\beta(\eta'\circ \eta))(m_{i})_{i\in\mathbb{N}} &= & (\eta'(\eta(m_{i}))_{i\in\mathbb{N}}\\
 & = & \beta(\eta')(\eta(m_{i}))_{i\in\mathbb{N}}\\
 & = & (\beta(\eta')\circ\beta(\eta))(m_{i})_{i\in\mathbb{N}}.
\end{eqnarray*}

\section{Details from Section~\ref{sec:representation_over_G}} \label{appendix_b}

\paragraph{Proof of Lemma \ref{lem:cat_equiv_general-G}.}
For a $\gapm$ morphism
$\left(\xi_{g}:M_{g}\rightarrow N_{g}\right)_{g\in G}$ 
be\-tween two $\gapms$ $\gapmname$ and $\othergapmname$, we define
\[
\alpha(\xi_{G}):\bigoplus_{g\in G}M_{g}\rightarrow\bigoplus_{g\in G}N_{g},(m_{g})_{g\in G}\mapsto(\xi_{g}(m_{g}))_{g\in G}.
\]

Recall that for $\xi_{G}=(\xi_{g})_{g\in G}$, a $\gapm$ morphism
between $\gapmname$ and $\othergapmname$, we define $\alpha(\xi_{G})$
as the map assigning to $(m_{g})\in\oplus M_{g}$ the value $(\xi_{g}(m_{g}))$.
Let us define $f:=\alpha(\xi_{G})$ for shorter notation. Again,
$f$ is a graded module morphism: it is a group homomorphism, as each
$\xi_{g}$ is, it clearly satisfies $f(M_{g})\subset N_{g}$, and it holds that 
 for $m=(m_{g})_{g\in G}$ and $h\in G$ we get
\begin{eqnarray*}
f(hm) & = & f\left(\left(\begin{cases}
hm_{\widetilde{g}}; & \exists\widetilde{g}\in G:h\star\widetilde{g}=g\\
0; & \text{otherwise}
\end{cases}\right)_{g\in G}\right)\\
 & = & \left(\left(\begin{cases}
\xi_{h\star\widetilde{g}}(hm_{\widetilde{g}}); & \exists\widetilde{g}\in G:h\star\widetilde{g}=g\\
0; & \text{otherwise}
\end{cases}\right)_{g\in G}\right)\\
 & = & \left(\left(\begin{cases}
h\xi_{\widetilde{g}}(m_{\widetilde{g}}); & \exists\widetilde{g}\in G:h\star\widetilde{g}=g\\
0; & \text{otherwise}
\end{cases}\right)_{g\in G}\right)\\
 & = & \left(\left(\begin{cases}
h\xi_{\widetilde{g}}(m_{\widetilde{g}}); & \exists\widetilde{g}\in G:h\star\widetilde{g}=g\\
0; & \text{otherwise}
\end{cases}\right)_{g\in G}\right)\\
 & = & h\left(\xi_{g}(m_{g})\right)_{g\in G} \\
 & = & hf(m),
\end{eqnarray*}
where the third equality comes from the property of $\gapm$ morphisms. 

For the functorial properties, it is clear that $\alpha$ maps the
identity $\gapm$ morphism to the identity of the corresponding graded module. If $\xi_{G}$, $\xi'_{G}$ are
$\gapm$ morphisms and $m$ as before, we calculate 
\begin{eqnarray*}
  (\alpha(\xi'_{\ast}\circ\xi_{\ast}))(m) & = & \left(\xi'_{g}(\xi_{g}(m_{g}))\right)_{g\in G}\\
 & = & \alpha(\xi'_{\ast})\left(\xi_{g}(m_{g})\right)_{g\in G}\\
 & = & (\alpha(\xi'_{\ast})\circ\alpha(\xi_{\ast}))(m).
\end{eqnarray*}

Vice versa, a morphism 
\[
\eta:\bigoplus_{g\in G}M_{g}\rightarrow\bigoplus_{g\in G}N_{g}
\]
in $\rgmodcat$ induces a homomorphism $\eta_{g}:M_{g}\rightarrow N_{g}$
for each $g\in G$, and these induced maps are compatible with multiplication
of each $h\in G$. Hence, the diagram
\[
\begin{xy}\xymatrix{M_{g}\ar[r]^{h}\ar[d]^{\xi_{g}} & M_{h\star g}\ar[d]^{\xi_{h\star g}}\\
N_{g}\ar[r]^{h} & N_{h\star g}
}
\end{xy}
\]
commutes and so, setting $\beta(\eta):=(\eta_{g})_{g\in G}$ yields
a $\gapm$ morphism be\-tween $\beta(\oplus_{g\in G}M_{g})$ and
$\beta(\oplus_{g\in G}N_{g})$. 

For functoriality, it is clear that $\beta$ maps the identity between graded modules
to the identity $\gapm$ morphism. For two graded $R[G]$-module morphisms
$\eta,\eta'$ and any $\left(m_{g}\right)_{g\in G}$ with $m_{g}\in M_{g}$
for all $g\in G$, we get 
\begin{eqnarray*}
 (\beta(\eta'\circ \eta))\left(m_{g}\right)_{g\in G} & = & \left(\eta'(\eta(m_{g})\right)_{g\in G}\\
 & = & \beta(\eta')\left(\eta(m_{g})\right)_{g\in G}\\
 & = & (\beta(\eta')\circ\beta(\eta))\left(m_{g}\right)_{g\in G}.
\end{eqnarray*}
Finally, the construction immediately implies that $\alpha\circ\beta$
equals the identity functor on $\rgmodcat$ and $\beta\circ\alpha$
equals the identity functor on $\gampcat$, hence $(\alpha,\beta)$ 
is an isomorphic pair of functors.

\section{Proving the ZC-Representation Theorem using Artin-Rees theory} \label{artin-rees-appendix}
We showed how to prove the Representation Theorem over general rings with unity and good monoids as index set. Artin-Rees theory is defined over commutative Noetherian rings with unity and filtrations of subsets of modules. Therefore it does not yield a proof for more general rings or, at least not immediately, for more general indexing monoids than $\N$. Since Artin-Rees theory is quoted in \cite{afra-book,ZC05} to provide a proof for Theorem \ref{thm:representation_theorem_naturals} (over commutative Noetherian rings with unity), we take a closer look at the connections between Artin-Rees theory and such a proof. To do this, we consider the following notion of filtrations of modules:

\begin{definition*}
	Let $A$ be a commutative Noetherian ring with unity. Let $I\subseteq A$ be an ideal, $M$ an $A$-module. An \emph{$I$-filtration of $M$} is a collection $(M_n)_{n\in\N}$ such that  $M=M_{0}\supseteq M_{1}\supseteq M_{2}\supseteq\ldots$ and $IM_{n}\subseteq M_{n+1}$ for all $n\in\N$. An $I$-filtration is called \emph{$I$-stable} if there is an $n_0\in\N$ such that $IM_{n}=M_{n+1}$ for all $n\geq n_0$.
\end{definition*}

Now, consider an ideal $I\subseteq A$ and an $I$-filtration $(M_n)_{n\in\N}$ of a finitely generated $A$-module $M$. Let $\overline{A}:=\overset{}{\underset{i\in\mathbb{N}}{\bigoplus}}I^{i}$, $\overline{M}:=\overset{}{\underset{i\in\mathbb{N}}{\bigoplus}}M_{i}$. We obtain a criterion for $(M_n)_{n\in\N}$ to be $I$-stable:

\begin{lemma*}[Criterion for stability] 
	Let $A$ be a commutative Noetherian ring with unity, $I\subseteq A$ an ideal. Let $M$ be a finitely generated $A$-module, $(M_n)_{n\in\N}$ an $I$-filtration of $M$. Then the following are equivalent: 
\begin{enumerate}
   \item $\overline{M}$  is a finitely generated  $\overline{A}$-module. 
   \item $(M_n)_{n\in\N}$  is  $I$-stable.
  \end{enumerate}
\end{lemma*}

This criterion helps to prove the famous Artin-Rees Lemma. It states that given a stable filtration of a finitely generated module $M$ and a submodule $N$ of $M$, then intersecting each member of the filtration with  $N$  again yields a stable filtration. The Artin-Rees Lemma can for instance be used to prove Krull's Intersection Theorem \cite{Eis95} and to study modules over local rings \cite{GP07}. For a proof of the above criterion, the Artin-Rees Lemma and the connection between the two lemmas, we refer to \cite{GP07}. 

\smallskip

Let us see how this helps to prove the ZC-Representation Theorem for a commutative Noetherian ring $R$ with unity. Setting  $A=R[t]$ and $I$ as the ideal generated by $t$, we obtain $\overline{A}=R[t]$. By Lemma \ref{lem:noethersch}, $R[t]$ is Noetherian and finite presentation and finite generation coincide not only for $R$-modules, but also for $R[t]$-modules. 

Let $\dapmname$  be a $\dapm$ over $R$ of finitely generated type. Consider the filtration $(\widetilde{M}_n)_{n\in\N}$ defined by $\widetilde{M}_0=\oplus_{i\in\N} M_{i}$ and $\widetilde{M}_n=(t^n)\widetilde{M}_{0}$ for $n>0$. By finite type assumption, there exists an $n_o\in\N$ such that $(t)\widetilde{M}_n=\widetilde{M}_{n+1}$ for all $n\geq n_0$. To use the above criterion, we have to ensure that $\oplus_{i\in\N} M_{i}$ is a finitely generated $R[t]$-module. We do not see how this could follow from Artin-Rees theory since finite generation is an assumption in most of the statements. A proof for the finite generation of  $\oplus_{i\in\N} M_{i}$ is the first part of our proof of Lemma \ref{lem:alpha_lemma}. 

Conversely, using the above criterion, finite generation of the $R[t]$-module $\oplus_{i\in\N} M_{i}$ directly implies $(t)$-stability of the filtration $(\widetilde{M}_n)_{n\in\N}$. Hence the corresponding $\dapm$ $\beta(\oplus_{i\in\N} M_{i})$ is of finite type. It is left to prove that each $M_i$ is finitely generated as an $R$-module. This can easily be seen by the first five lines of our proof of Lemma \ref{lem:comp_are_finitely_presented}.  

\bibliography{algpersmod}

\newpage{} 
\end{document}